\documentclass{article}
\usepackage{arxiv}

\usepackage[normalem]{ulem} 

\usepackage[utf8]{inputenc} 
\usepackage[T1]{fontenc}    
\usepackage{hyperref}       
\usepackage{url}            
\usepackage{booktabs}       
\usepackage{amsfonts}       
\usepackage{microtype}      
\usepackage{color}
\usepackage{lipsum}
\usepackage{amsthm}
\usepackage{indentfirst}
\usepackage{graphicx}
\usepackage{epsfig,epsf,latexsym,subfigure}
\usepackage{float}
\usepackage{epstopdf}
\usepackage{fourier}
\usepackage{bm}
\usepackage{bbm} 
\usepackage{mathtools}
\usepackage{latexsym,enumerate}
\usepackage{tabularx}
\usepackage{capt-of}
\usepackage{cases}
\usepackage{enumitem} 
\usepackage{algorithm}
\usepackage{algpseudocode}
\usepackage{fourier}

\DeclareMathAlphabet{\mathcal}{OMS}{cmsy}{m}{n}
\SetMathAlphabet{\mathcal}{bold}{OMS}{cmsy}{b}{n}

\newcommand{\comment}[1]{}

\newcommand{\BEA}{\begin{eqnarray}}
\newcommand{\EEA}{\end{eqnarray}}

\newcommand{\BR}{\mathbb{R}}

\newtheorem{theo}{Theorem}[section]
\newtheorem{prop}{Proposition}[section]

\newtheorem{scheme}{Scheme}[section]

\title{Geometric local parameterization for solving Hele-Shaw problems with surface tension}

\author{
  Zengyan Zhang and Wenrui Hao \\
  Department of Mathematics \\
  The Pennsylvania State University, University Park, PA 16802, USA\\
  \texttt{zzz5527@psu.edu} and \texttt{wxh64@psu.edu} \\
  \And
 John Harlim \\
  Department of Mathematics, \\ Institute for Computational and Data Sciences \\
  The Pennsylvania State University, University Park, PA 16802, USA\\
  \texttt{jharlim@psu.edu} \\
}

\begin{document}

\maketitle

\begin{abstract}
In this work, we introduce a novel computational framework for solving the two-dimensional Hele-Shaw free boundary problem with surface tension. The moving boundary is represented by point clouds, eliminating the need for a global parameterization. Our approach leverages Generalized Moving Least Squares (GMLS) to construct local geometric charts, enabling high-order approximations of geometric quantities such as curvature directly from the point cloud data. This local parameterization is systematically employed to discretize the governing boundary integral equation, including an analytical formula of the singular integrals. We provide a rigorous convergence analysis for the proposed spatial discretization, establishing consistency and stability under certain conditions. The resulting error bound is derived in terms of the size of the uniformly sampled point cloud data on the moving boundary, the smoothness of the boundary, and the order of the numerical quadrature rule. Numerical experiments confirm the theoretical findings, demonstrating high-order spatial convergence and the expected temporal convergence rates. {\color{black}The method's effectiveness is further illustrated through simulations of complex initial shapes, including interfaces driven by anisotropic surface tension, which correctly evolve towards circular equilibrium states under the influence of surface tension, highlighting the versatility of the method for complex geometry-dependent interface dynamics.}
\end{abstract}

\paragraph{Keywords.} Hele-Shaw problem, generalized moving least squares, boundary integral method, singular integral equation

\paragraph{MSCcodes.}  65M12 35R35 45E05

\section{Introduction}
Many natural phenomena in biology, physics, and materials science are governed by systems of partial differential equations (PDEs) with free (moving) boundaries \cite{F1,FH,HCF,HF,HHHS}, including plaque growth in cardiovascular disease \cite{FH,HF}, tumor growth in cancer \cite{cristini2003nonlinear,HHHS,macklin2006improved}, solidification and melting in materials science \cite{alexiades2018mathematical,gupta2017classical}, and moving fluid interfaces in physics \cite{mohammad2022review,vogel1996life}. Among these, the Hele-Shaw problem with surface tension stands out as a widely studied model due to its broad applications across both physical and biological systems \cite{F1,friedman2012variational}. Since the pioneering experiments of Hele-Shaw, who confined a fluid between two closely spaced plates \cite{HSH}, and the seminal work of Saffman and Taylor in 1958 \cite{ST}, this problem has attracted extensive theoretical and experimental attention.

From a mathematical perspective, the Hele-Shaw problem can be studied both analytically and numerically, with an emphasis on understanding solution structures and their dynamics \cite{chen2003free,dibenedetto1984ill}. Over the past few decades, extensions of the Hele-Shaw model incorporating surface tension have arisen from physical and biological applications \cite{constantin1993global,F1}, motivating the development of theoretical frameworks and nonlinear simulation techniques to investigate steady-state solutions. While PDE theory provides valuable insights in certain special cases, in-depth studies often rely on large-scale numerical simulations to compute steady states, trace bifurcations, and analyze stability \cite{HHHMS,HHHS,HHHLSZ}. For example, bootstrapping methods combined with multi-grid and domain decomposition techniques, along with homotopy approaches, have been used to compute multiple steady-state solutions of the generalized Hele-Shaw problem \cite{HHHMS,HHHS,HHS1}. Moreover, adaptive homotopy tracking methods enable the detection of bifurcation points and exploration of global solution structures guided by PDE theory \cite{HHHLSZ,HHHS}. These methods have also been successfully applied to complex biological systems, such as tumor growth models and the evaluation of cardiovascular disease risk \cite{HCF,HF}.

Despite these advances, developing efficient numerical algorithms for complex PDE systems with free boundaries remains a central challenge. For example, rigorous convergence and error analyses have been established for boundary integral methods applied to simplified Hele–Shaw problems {\color{black} without accounting for surface tension \cite{HHLS}}. Specifically, the scheme attains first-order accuracy in time ($L^\infty$ norm) and $\Delta \theta^\varrho$ ($\varrho<1$) in space, relying on a global parameterization in the $\theta$ direction, which limits its ability to handle complex geometries. Furthermore, singular integration in the boundary integral is not addressed in their analysis.  {\color{black} These difficulties are further amplified when the geometry undergoes complex topological changes, such as in interfaces driven by anisotropic surface tension, where the interface may evolve from a multi-fold configuration. Classical approaches, including front-tracking  and level set methods, have been successfully used to track such evolving geometries. However, they often rely on global parameterizations or require frequent remeshing to preserve numerical stability, and accurate computation of curvature can become challenging in the presence of strong anisotropy. These considerations motivate the development of numerical methods based on local geometric representations that can flexibly adapt to complex interface geometries.} 


In this paper, we introduce an efficient numerical method for approximating the solution to Hele-Shaw problems with surface tension on complex geometry, with theoretical guarantees. While the formulation can be extended to arbitrary dimensions, we focus on two-dimensional problems with a one-dimensional smooth complex boundary. By ``complex boundary'', we refer mathematically to manifolds that cannot be globally parameterized, i.e., whose atlas has more than one chart. Since the boundary is unknown, that is, it is represented only by a finite number of point cloud data lying on a smooth manifold (or boundary curve), we locally parameterize this manifold with Generalize Moving Least Squares (GMLS), a method with solid theoretical foundations \cite{mirzaei2012generalized}.
Numerically, this local parameterization effectively represents functions on complex geometries \cite{gross2020meshfree,jiang2024generalized,liang2013solving}. Our main technical contribution is a systematic application of this local parameterization to discretize the singular boundary integral equation while respecting the underlying manifold. We further support the method with a theoretical study demonstrating the convergence of the scheme in the limit of a large number of uniformly sampled point cloud data.

The remainder of this paper is organized as follows. In Section~\ref{sec2}, we give a brief overview of the Hele-Shaw problem with surface tension and its boundary integral formulation. Section~\ref{sec3} presents an overview of the Generalized Moving Least Squares (GMLS) for approximating boundary curvature, the proposed numerical formulation for solving the free boundary problem, and the corresponding theoretical analysis. In Section~\ref{sec4},
we rigorously analyze the convergence and error of the proposed discretization. In Section~\ref{sec5},
we demonstrate the effectiveness of the proposed method on several test problems, ranging from simple to complex initial boundaries. We conclude the paper with a summary and discussion in Section~\ref{sec6}. We accompany this paper with appendices that reported some detailed calculaions of the proposed discretization, additional proofs of some propositions, and additional numerical results. 

\section{Overview of the Hele–Shaw problem with surface tension and its boundary integral formulation}\label{sec2}

Consider a tissue or material that behaves as a porous medium. The velocity field $\mathbf V$ is described by Darcy's law, 
$\mathbf{V} = -\frac{\sigma}{\mu} \nabla p,$
where $p$ denotes the pressure, $\sigma$ is the permeability of the medium, and $\mu$ is the dynamic viscosity of the fluid. Combining with the conservation of mass, $\nabla \cdot \mathbf V=f$, where $f:\Omega(t) \to \BR$ represents distributed sources or sinks, leads to the following PDE system with a free boundary,
\begin{equation}\label{hele-shaw}
\left\{
\begin{array}{rcll}
-\frac{\sigma}{\mu}\Delta p &=& f & \text{in }\, \Omega(t), \\
p &=& \tau\kappa & \text{on }\, \Gamma(t), \\
\frac{\sigma}{\mu}\frac{\partial p}{\partial \mathbf n} &=& -V_n & \text{on }\, \Gamma(t).
\end{array}
\right.
\end{equation}
Here, \( \Omega(t) \) denotes the fluid-filled region, and \( \Gamma(t) \) denotes its moving boundary. In this paper, we consider a bounded two-dimensional domain $\Omega(t) \subset \mathbb{R}^2$ with a closed curve boundary $\Gamma(t)$ at any fixed time $t\geq 0$. In \eqref{hele-shaw}, \( \kappa \) denotes the mean curvature (for instance, $\kappa=R(t)^{-1}$ if $\Omega(t)$ is a disk of radius $R(t)$) {\color{black}\cite{hou1997hybrid,li2007rescaling}}.  
The boundary condition $p=\tau\kappa$ models the effect of surface tension: the pressure jump across the interface is proportional to the local curvature, which acts to stabilize and regularize the evolving boundary. In classical Hele–Shaw problems without surface tension, the pressure on the boundary {\color{black} is typically prescribed as a constant \cite{aitchison1985computation} or a given potential function \cite{blank2009hele}.} The inclusion of surface tension distinguishes the present model, as the boundary pressure is not known a priori but instead couples dynamically to the evolving geometry through curvature. This feature introduces additional analytical and numerical challenges and is a central focus of this work.  
For simplicity, we normalize parameters and set $\sigma=\mu$ and $\tau=1$ throughout the remainder of the paper.
We assume the kinematic boundary condition on the free boundary $\Gamma(t)$, which states that the boundary moves in accordance with the velocity $\mathbf V$, such that its component in the direction of the outward normal \(\mathbf n \) is given by
\begin{equation*}
V_n=\mathbf V\cdot\mathbf n=-\frac{\partial p}{\partial \mathbf n} \quad\text{on}~\Gamma(t),
\end{equation*}
with $\mathbf n$ pointing out of $\Omega(t)$. 

In this paper, we consider the case $f=0$, which corresponds to the classical Hele-Shaw free boundary problem. This allows us to integrate the governing equations on the free boundary. Denote $G(\mathbf x,\mathbf y)$ as the Green's function for the Laplacian operator, $-\Delta$. Specifically, \( G(\mathbf{x}, \mathbf{y}) = -\frac{1}{2\pi} \ln\|\mathbf{x} - \mathbf{y}\| \) for two-dimensional cases. Then from the Green's third identity \cite{friedman1963generalized,kress1999linear}, we have
\[-\int_\Omega G(\mathbf x,\mathbf y)\Delta p(\mathbf y) dV_{\mathbf y}-p(\mathbf x)=-\int_{\Gamma}\Big(G(\mathbf x,\mathbf y)\frac{\partial p(\mathbf y)}{\partial\mathbf n(\mathbf y)}-p(\mathbf y)\frac{\partial G(\mathbf x,\mathbf y)}{\partial\mathbf n(\mathbf y)}\Big) dS_{\mathbf y}, \quad\forall \mathbf x\in\Omega(t) \]

Setting $\Delta p=0$, incorporating the jump condition and the boundary condition, we arrive at 
\[
\frac{\kappa(\mathbf{x})}{2} =\int_{\Gamma}\Big(-G(\mathbf{x},\mathbf{y})V_n(\mathbf y) -\kappa(\mathbf{y})\frac{\partial G(\mathbf{x},\mathbf{y})}{\partial \mathbf{n}(\mathbf y)}\Big)dS_\mathbf{y},\quad \mathbf{x},\mathbf{y} \in \Gamma(t),
\] 
where $\mathbf{n}=\frac{\nabla \Gamma}{|\nabla \Gamma|}$ with $\mathbf n(\mathbf y)$ referring to the unit outer normal vector at $\mathbf y\in\Gamma(t)$, and $\kappa=\nabla\cdot\mathbf{n}$.

Therefore, we have reformulated the classical Hele-Shaw free boundary problem \eqref{hele-shaw} with $f=0$ as the following system,
\begin{subnumcases}{\label{BIM}}
\int_{\Gamma}G(\mathbf{x},\mathbf{y})V_n(\mathbf y) dS_\mathbf{y}=-\frac{\kappa(\mathbf{x})}{2} -\int_\Gamma \kappa(\mathbf{y})\frac{\partial G(\mathbf{x},\mathbf{y})}{\partial \mathbf{n}(\mathbf y)}dS_\mathbf{y}  ,\quad \mathbf{x},\mathbf{y} \in \Gamma(t),\label{BIE}\\
\frac{d\mathbf x}{dt} =  V_n(\mathbf x) \mathbf n(\mathbf x), \quad \mathbf x\in\Gamma(t),\label{velocity}
\end{subnumcases}
where the solution to the boundary integral equation (BIE) \eqref{BIE} gives the unknown normal velocity $V_n$ on $\Gamma(t)$, which is represented by point clouds without parameterization information, and the kinematic condition \eqref{velocity} enables the free boundary to evolve forward in time. {\color{black} In general, $\frac{d\mathbf{x}}{dt}=V_n({\bf x})\mathbf{n}({\bf x})+V_t({\bf x})\mathbf{t}({\bf x})$, where $V_t$ and $\mathbf{t}$ denote the tangential velocity and the unit tangent vectors, respectively. Since for any choice of $V_t$, the Neumann boundary condition in \eqref{hele-shaw} is satisfied, we set $V_t=0$ for simplicity.

While such a choice can lead to clustered points along the boundary, as needed, we will re-mesh the interface so that the boundary points remain approximately equally spaced along the boundary. The details are provided in Section~\ref{sec5}.}

\section{Numerical approximation for the free boundary problems}\label{sec3}

In this section, we introduce our approach for approximating the solution to \eqref{BIM} through a local representation of the manifold, which naturally allows one to approximate the integration on the boundary directly. We begin with a brief overview of the Generalized Moving Least Squares (GMLS) method, which constructs local manifold charts and approximates higher-order geometric quantities, such as the curvature $\kappa$ along moving boundaries. Next, we present the quadrature rules used to discretize the boundary integral equation \eqref{BIE} in space, yielding a corresponding system of algebraic equations. Solving this system in conjunction with a time discretization to \eqref{velocity} enables us to track the evolution of the boundary $\Gamma(t)$ over time.

\subsection{Generalized Moving Least Squares approximation of the curvature}
Given a set of point cloud data $X=\{\mathbf x_i\}_{i=1}^N\subset\mathbb{R}^n$ sampled from a manifold $\Gamma$ of dimension $d$, we denote the set of $k$-nearest neighbors for any point $\mathbf x_i$ in $X$ by $K(i)=\{\mathbf x_{i_1},\dots,\mathbf x_{i_k}\}$. In this notation, $\mathbf x_i=\mathbf x_{i_1}$. We can obtain a local estimation of the tangent space at each point by the classical local SVD method stated as follows,
\begin{enumerate}
    \item Construct the distance matrix $\mathbf D_{i}:=[\mathbf D_{i_1},\dots,\mathbf D_{i_k}]\in\mathbb {R}^{n\times k}$, where $k>d$ and $\mathbf D_{i_j}:=\mathbf x_{i_j}-\mathbf x_i$, $j=1,\dots,k$.
    \item Obtain an orthonormal basis $\widetilde{\mathbf T}_i=\{\tilde{\mathbf t}_i^{(1)},\dots,\tilde{\mathbf t}_i^{(d)}\}$ for the approximated tangent space $\widetilde{T_{\mathbf x_i}\Gamma}$, by taking the singular value decomposition (SVD) of $\mathbf D_i$.
\end{enumerate}

To attain a higher-order approximation of the local tangent space, we consider the Generalized Moving Least Square method \cite{gross2020meshfree,jiang2024generalized,liang2013solving,mirzaei2012generalized}. Generally, we use an intrinsic polynomial to approximate smooth function $g:\Gamma \to \BR$ over the neighborhood of $\mathbf x_i$ which is the optimal solution for the following least-squares problem,
\[\min_{\hat g\in\mathbb{P}_{\mathbf x_i}^{\ell,d}}\sum_{j=1}^k\Big(g(\mathbf x_{i_j})-\hat g(\mathbf x_{i_j})\Big)^2,\]
where $\mathbb{P}_{\mathbf x_i}^{\ell,d}$ is the space of intrinsic polynomial with degree up to $\ell$ at the point $\mathbf x_i\in\Gamma\subset\mathbb{R}^d$. We refer to $\hat{g}$ as the Generalized Moving Least Squares (GMLS) estimate of $g$. For convenience in the discussion below, we present an error bound (cast in our notation) for this least-squares fit, as derived in \cite{jiang2024generalized}.

\begin{prop}\label{prop3.1}
Let $X \subset \Gamma$ be a set of $N$ uniformly sampled i.i.d. data from a $d$-dimensional manifold $\Gamma$. Assume that $g \in C^{\ell+1}(\Gamma)$. With probability higher than $1-\frac{1}{N}$,
\[
\left|D^\alpha g - D^\alpha \hat{g} \right|  = \mathcal{O} \left( \left(\frac{\log N}{N} \right)^{\frac{\ell+1-|\alpha|}{d}}\right), 
\]
as $N\to \infty$. Here, $D^\alpha$ denotes the multiindex derivative of order $|\alpha|$, and the constant in the big-$\mathcal{O}$ notation can depend on $d$, but it is independent of $N$. 
\end{prop}

In the tangent space approximation, the function $g$ will be the local parameterization of the manifold, with the tangent space represented by its Jacobian. To illustrate the idea, we consider the case $d=1$ and $n=2$ in our paper. Recall that from the local SVD method, we obtain $\{\tilde{\mathbf t}_i,\tilde{\mathbf n}_i$\}, where $\tilde{\mathbf t}_i$ is the approximate tangent vector at $\mathbf x_i$, and $\tilde{\mathbf n}_i$ is an approximation of the normal direction at $\mathbf x_i$. We first project $\mathbf{D}_i$ onto the estimated tangent space,
\[\tilde T(i)=\{\tilde{\mathbf t}_i\cdot(\mathbf x_{i_1}-\mathbf x_{i_1}),\dots,\tilde{\mathbf t}_i\cdot(\mathbf x_{i_k}-\mathbf x_{i_1})\}\subset \widetilde{T_{\mathbf x_i}\Gamma},\]

and denote the normal components as
\[\tilde N(i)=\{\tilde{\mathbf n}_i\cdot(\mathbf x_{i_1}-\mathbf x_{i_1}),\dots,\tilde{\mathbf n}_{i}\cdot(\mathbf x_{i_k}-\mathbf x_{i_1})\}.\]

Consider a polynomial $\tilde p_i:\widetilde{T_{\mathbf x_i}\Gamma}\rightarrow\mathbb{R}$ defined by
\[\tilde p_i(\tilde s)=\tilde\alpha_{i,1} \tilde s+\tilde\alpha_{i,2} \tilde s^2+\cdots+\tilde\alpha_{i,\ell} \tilde s^\ell,\]
where $\tilde s\in \tilde T(i)$ and the coefficients are obtained by a least squares fit to the data in $\tilde T(i)$ with labels given by $\tilde N(i)$, i.e., 
\[(\tilde\alpha_{i,1},\dots,\tilde\alpha_{i,\ell})=\arg\min \sum_{j=1}^k\Big(\tilde p_i(\tilde s)-\tilde{\mathbf n}_{i}\cdot(\mathbf x_{i_j}-\mathbf x_{i_1})\Big)^2.\]

Therefore, we can define a local coordinate chart for the manifold near the point $\mathbf x_i$ using the embedding map $\tilde\iota_i:\widetilde{T_{\mathbf x_i}\Gamma}\rightarrow\mathbb{R}^2$,
\begin{equation}\label{svd-map}
\tilde\iota_i(\tilde s)=\mathbf x_i+\tilde{\mathbf t}_i \tilde s+\tilde{\mathbf n}_i \tilde p_i(\tilde s).
\end{equation}

With this GMLS approximation, we also inherit the estimated tangent space from the Jacobian of $\tilde\iota_i$ defined in \eqref{svd-map}. The GMLS approximation of the tangent vectors over the neighborhood of $\mathbf x_i$ is
\begin{equation}\label{gmls_tangent}
\mathbf{\hat{t}}_i(\tilde s)=\tilde{\mathbf t}_i+\tilde{\mathbf n}_i\tilde p'_i(\tilde s),
\end{equation}
with the estimated tangent vector at $\mathbf x_i$ denoted by $\mathbf{\hat{t}}_i=\mathbf{\hat{t}}_i(0)$, as a GMLS approximation to the underlying tangent vector $\mathbf t(\mathbf x_i)$. We should point out that the additional normal correction term in \eqref{gmls_tangent} yields an improved approximation compared to the SVD approximation, $\mathbf{\tilde{t}}_i$. For uniformly sampled random data of size $N$, it was shown that the SVD approximation converges at order $N^{-1}$, provided that a sufficiently large number of $k$-nearest points is used (see Remark 9 in \cite{harlim2023radial}). According to Proposition~\ref{prop3.1}, if the manifold is $C^{\ell+1}$, the GMLS estimate $\tilde{\iota}_i(0)$ of the true parameterization $\mathbf x_i=\iota(s)$ for some $s$ has an error of order $N^{-\ell-1}$, ignoring the $\log(N)$ factor. Then the tangent vector estimate $\mathbf{\hat{t}}_i$ converges at rate $N^{-\ell}$, which indicates an improved estimate for $\ell >1$ and nonzero $\tilde{p}'_i(0)=\tilde{\alpha}_{i,1}$. This fact suggests that the estimation for the local parameterization (as well as the local tangent vector) can be refined by applying the GMLS procedure iteratively on the local coordinates $(\mathbf{\hat{t}}_i,\mathbf{\hat{n}}_i)$, where $\mathbf{\hat{n}}_i$ is a unit normal vector pointing outward from $\Omega(t)$ such that $\mathbf{\hat{t}}_i\cdot \mathbf{\hat{n}}_i = 0.$  Specifically, the GMLS approximation of the manifold can be repeated near the point $\mathbf x_i$ by employing the embedding map $\iota_i:T_{\mathbf x_i}\Gamma\rightarrow \mathbb{R}^2$,
\begin{equation}\label{gmls-map}
\iota_i(s)=\mathbf x_i+\mathbf{\hat{t}}_is+\mathbf{\hat{n}}_i p_i(s),
\end{equation}
where $p_i(s)=\alpha_{i,1}s+\alpha_{i,2} s^2+\cdots+\alpha_{i,\ell} s^\ell$ and $\alpha_{i,1},\alpha_{i,2},\cdots,\alpha_{i,\ell}$ are obtained via GMLS fitting using the data pairs from $s\in T(i)=\{\mathbf{\hat{t}}_i\cdot(\mathbf x_{i_1}-\mathbf x_{i_1}),\dots,\mathbf{\hat{t}}_i\cdot(\mathbf x_{i_k}-\mathbf x_{i_1})\}\subset T_{\mathbf x_i}\Gamma$ and $N(i) =\{\mathbf{\hat{n}}_i\cdot(\mathbf x_{i_1}-\mathbf x_{i_1}),\dots,\mathbf{\hat{n}}_i\cdot(\mathbf x_{i_k}-\mathbf x_{i_1})\}$. We iterate this process multiple times, updating tangent vector estimate with the Jacobian of $\iota_i(0)$, and stop when the polynomial linear coefficient $p'_i(0)$ falls below a desired tolerance as a stopping criterion, as further iterations provide negligible improvement. 
In our numerical simulations, we take this stopping criterion to be  $p'_i(0)=\alpha_{i,1} \approx 10^{-12}$, and denote the resulting estimated local tangent and normal vectors as $(\mathbf{\hat{t}}_i,\mathbf{\hat{n}}_i)$ for notational simplicity. As a result, $p_i(s)$ effectively has no linear term, i.e., $p_i(s)=\alpha_{i,2} s^2+\cdots+\alpha_{i,\ell} s^\ell$.
The overall GMLS method is summarized in Algorithm~\ref{alg:gmls}.

\begin{algorithm}[htbp]
\caption{Generalized Moving Least Squares Approximation for the Local Parametrization}\label{alg:gmls}
\begin{algorithmic}[1]
\Require \parbox[t]{\dimexpr0.95\linewidth-\algorithmicindent}{%
A set of (distinct) nodes $X=\{\mathbf x_i\}_{i=1}^N\subset\Gamma$, $k(\ll N)$ nearest neighbors of the point $\mathbf x_i$ in the stencil $K(i)=\{\mathbf x_{i_j}\}_{j=1}^k$, and the approximate bases $\{\tilde{\mathbf t}_i,\tilde{\mathbf n}_i\}_{i=1}^N$ of the local tangent and normal spaces via the classical local SVD method. }
\For{$i\in\{1,\cdots,N\}$}
\State Construct the matrix $\Phi\in\mathbb{R}^{k,\ell}$ and the vector $\Psi\in\mathbb{R}^{k}$ with
\[\Phi_{j,r}=\Big(\tilde{\mathbf t}_i\cdot(\mathbf x_{i_j}-\mathbf x_{i_1})\Big)^r, \quad \Psi_j=\tilde{\mathbf n}_i\cdot(\mathbf x_{i_j}-\mathbf x_{i_1}),\quad j=1,\cdots,k,\quad r=1,\cdots,\ell.\]
\State Obtain the coefficients $\tilde{\alpha}=\begin{bmatrix}\tilde{\alpha}_{i,1},\tilde{\alpha}_{i,2}, \ldots,\tilde{\alpha}_{i,\ell}\end{bmatrix}^\top$ of the intrinsic polynomial $\tilde{p}_i(\tilde s)$ with
\[\tilde{\mathbf \alpha}=(\Phi^\top\Phi)^{-1}\Phi^\top\Psi.\]
\State \parbox[t]{\dimexpr\linewidth-\algorithmicindent}{%
Construct the improved approximation of the local tangent vectors with
\[\hat{\mathbf t}_i(\tilde s)=\tilde{\mathbf t}_i+\tilde{\mathbf n}_i\tilde{p}'_i(\tilde s),\]
 and the estimate local normal vectors $\hat{\mathbf n}_i$ can be constructed by rotating $\hat{\mathbf t}_i$ and choosing the orientation so that it points outward from $\Omega(t)$.}
 \State The following steps improve the estimation, set $\alpha_{i,1} = \tilde{\alpha}_{i,1}$.
 \While{$p_i'(0)=\alpha_{i,1}>10^{-12}$}
\State \parbox[t]{\dimexpr0.95\linewidth-\algorithmicindent}{%
Update the vectors $\tilde{\mathbf t}_i,\tilde{\mathbf n}_i \leftarrow \hat{\mathbf t}_i,\hat{\mathbf n}_i$

Repeat step 2.

Repeat step 3 above with output denoted as
$\alpha=\begin{bmatrix}{\alpha}_{i,1},{\alpha}_{i,2}, \ldots,\alpha_{i,\ell}\end{bmatrix}^\top$ of the intrinsic polynomial $p_i(s)$ with
\[{\mathbf \alpha}=(\Phi^\top\Phi)^{-1}\Phi^\top\Psi.\]

Construct the improved approximation as in Step 4 above, with output 
\[\hat{\mathbf t}_i(s)=\hat{\mathbf t}_i+\hat{\mathbf n}_ip'_i(s),\]
 and the estimate local normal vectors $\hat{\mathbf n}_i$ can be constructed by rotating $\hat{\mathbf t}_i$ and choosing the orientation so that it points outward from $\Omega(t)$.
}

\EndWhile
\EndFor
\Ensure{The intrinsic polynomial approximation for the local parametrization $\{p_i(s)\}_{i=1}^N$.}
\end{algorithmic}
\end{algorithm}

Let $\gamma_i(s)=(s,p_i(s))$ be the coordinates in \eqref{gmls-map}, which is a local parametric representation of $\Gamma(t)$ near the point $\mathbf x_i$. Then, the approximate curvature near the point $\mathbf x_i$ is given by,
\begin{equation}\label{cur_gmls}
\kappa_i(s) =-\frac{\text{det}(\gamma',\gamma'')}{||\gamma'||^3}= -\frac{p_i''(s)}{\Big(1+\big(p_i'(s)\big)^2\Big)^{\frac{3}{2}}},
\end{equation}
with the estimated curvature at $\mathbf x_i$ denoted by $\kappa_i=\kappa_i(0)=-2\alpha_{i,2}$. Here we point out that the negative sign in \eqref{cur_gmls} is due to $\hat{\mathbf n}_i$ pointing out of $\Omega(t)$ (for example, $\kappa_i=1$ and $\alpha_{i,2}<0$ in the case of a unit circle). According to Proposition~\ref{prop3.1}, the GMLS approximation of the curvature converges at rate $\left(N^{-1}\log N\right)^{\ell-1}$ for uniformly sampled random data of size $N$. The convergence result for the unit circle case is demonstrated in Figure \ref{fig:curvature} for $\ell =3$ and $4$, where the average curvature error is defined as \begin{equation}e_{\kappa}=\sqrt{\frac{1}{N}\sum_{i=1}^N(\kappa_i-1)^2}.\label{rmsecurvature}\end{equation}
The numerical results suggest that the errors are slightly faster than the theoretical bound by a factor of $\big(\log N\big)^{\ell-1}$ for $\ell=3$ and $4$.
\begin{figure}[htbp]
\center
\includegraphics[width=0.5\linewidth]{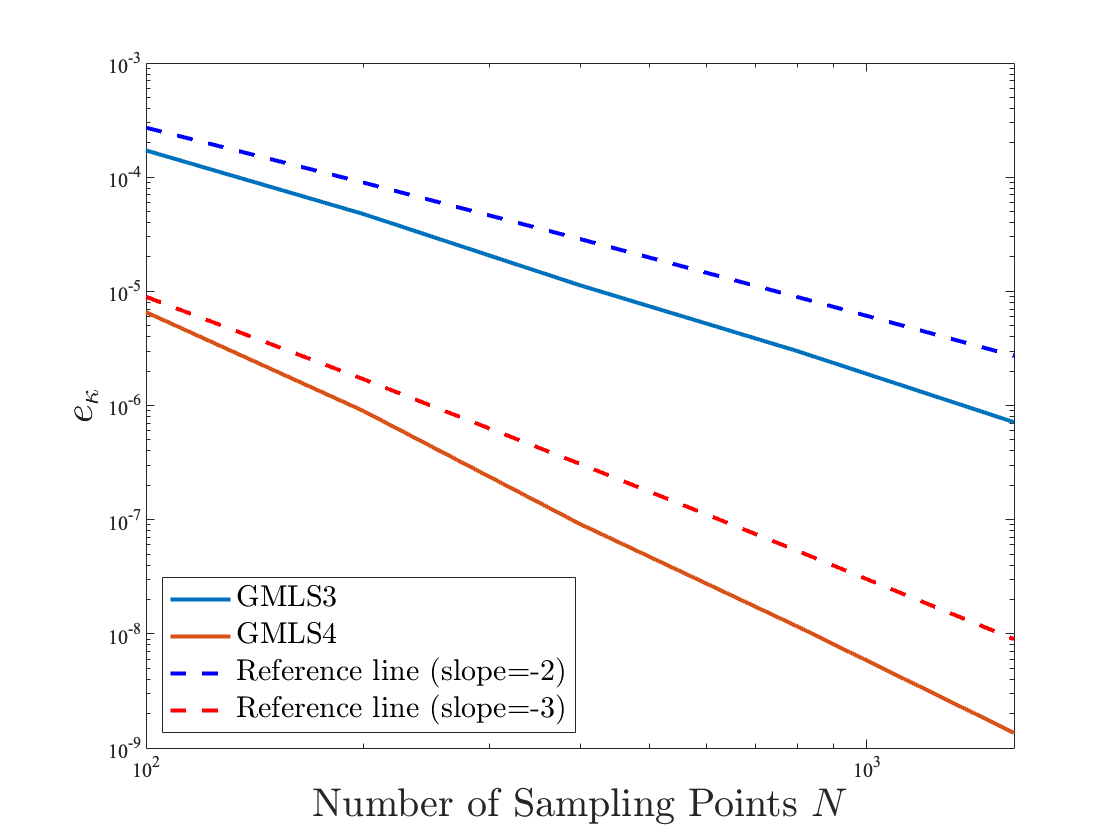}
\caption{Root-Mean-Square errors of the curvature as defined in \eqref{rmsecurvature} as functions of $N$ for the GMLS estimates with $\ell=3$ and 4.}
\label{fig:curvature}
\end{figure}

\subsection{Quadrature rule to approximate the boundary integral equation}\label{quadrature_rule}

In this section, we discretize the BIE \eqref{BIE} in space, which can be written as a linear system $AV_n(\mathbf{x})=b(\mathbf{x})$, with 
\begin{eqnarray}
A V_n(\mathbf{x}) &=& \int_{\Gamma} G(\mathbf x,\mathbf y) V_n(\mathbf y) dS_{\mathbf y}, \label{eq:AVn} \\ b (\mathbf{x})&=&-\frac{\kappa(\mathbf x)}{2}-\int_\Gamma \kappa(\mathbf y)\nabla G(\mathbf x,\mathbf y)\cdot \mathbf{n}(\mathbf y)dS_{\mathbf y}. \label{eq:b}
\end{eqnarray}

We approximate the boundary integral in \eqref{eq:AVn}-\eqref{eq:b} with the approximate local chart obtained from the GMLS approximation in \eqref{gmls-map}. Subsequently, we employ a quadrature rule to derive a discrete approximation of the linear system.

Locally, near $\mathbf{x}_j$, we approximate
\BEA dS_{\mathbf y} = \sqrt{|\iota'(s)|}ds \approx \sqrt{|\iota_j'(s)|}ds = \sqrt{1+\big(p_j'(s)\big)^2}ds, \label{eq:approxJac}\EEA
where we use the local parameterization $\mathbf y=\iota_j(s)$ in \eqref{gmls-map} near $\mathbf{x}_j$ for all $j=1,\ldots, N$. In terms of the local coordinate system through a change of variables, at $\mathbf x=\mathbf x_i\in\Gamma(t)$,
\BEA\label{eqn:A-int}
AV_n(\mathbf{x}_i)&\approx& \int_{\Delta s_{-i}}^{\Delta s_i} G\big(\mathbf x_i,\iota_i(s)\big)V_n\big(\iota_i(s)\big)\sqrt{1+\big(p_i'(s)\big)^2}ds 
+\sum_{j\neq i-1,i}\int_0^{\Delta s_j} G\big(\mathbf x_i,\iota_j(s)\big)V_n\big(\iota_j(s)\big)\sqrt{1+\big(p_j'(s)\big)^2}ds,\label{eq:AVdetail}
\EEA
 where $\Delta s_i={\hat{\mathbf{t}}_i}^\top(\mathbf x_{i+1}-\mathbf x_i)>0$, $\Delta s_{-i}={\hat{\mathbf{t}}_i}^\top(\mathbf x_{i-1}-\mathbf x_{i})<0$. 

Since $G(\mathbf x,\mathbf y)=-\frac{1}{2\pi}\ln\|\mathbf x-\mathbf y\|$, 
\BEA
G\big(\mathbf x_i,\iota_i(s)\big)=-\frac{1}{2\pi}\ln\left\|\mathbf x_i-\big(\mathbf x_i+\mathbf{\hat{t}}_i s+\mathbf{\hat{n}}_i p_i(s)\big)\right\| = -\frac{1}{2\pi}\ln \sqrt{s^2+ p_i^2(s)},
\EEA
it is clear that the first integral in \eqref{eq:AVdetail} is singular at $s=0$, because $\mathbf \iota_i(0)=\mathbf x_i$. Then we rewrite the first term in \eqref{eq:AVdetail} as
\BEA
       \int_{\Delta s_{-i}}^{\Delta s_i} G\big(\mathbf x_i,\iota_i(s)\big)V_n\big(\iota_i(s)\big)\sqrt{1+\big(p_i'(s)\big)^2}ds
       &=& \underbrace{-\frac{1}{2\pi}\int_{\Delta s_{-i}}^{\Delta s_i} \ln |s|V_n\big(\iota_i(s)\big)\sqrt{1+\big(p_i'(s)\big)^2}ds}_{=I_S} \notag \\ && \underbrace{-\frac{1}{4\pi}\int_{\Delta s_{-i}}^{\Delta s_i} \ln\left(1+\frac{p_i^2(s)}{s^2}\right)V_n\big(\iota_i(s)\big)\sqrt{1+\big(p_i'(s)\big)^2}ds}_{I_{NS}}. 
\label{eq_sing}
\EEA
Notice that $\frac{p_i(s)}{s}\Big|_{s=0} = 0$. Thus, \eqref{eq_sing} has a singular term $I_S$ that involves $\ln|s|$. Accordingly, we can rewrite \eqref{eq:AVdetail} as
\BEA
AV_n(\mathbf{x}_i) \approx I_S + I_{N},\label{eq_approx_Ax}
\EEA
where the nonsingular part $I_N$ consists of $I_{NS}$ and the second integral term in \eqref{eq:AVdetail},
\BEA
I_N = I_{NS} + \sum_{j\neq i-1,i}\int_0^{\Delta s_j} G\big(\mathbf x_i,\iota_j(s)\big)V_n\big(\iota_j(s)\big)\sqrt{1+\big(p_j'(s)\big)^2}ds,\label{I_N}
\EEA
which can be further discretized using any standard quadrature rule.

To handle the singular term $I_S$, the function $\psi_i(s):=V_n\big(\iota_i(s)\big)  \sqrt{1 + \big(p_i'(s)\big)^2}$ is approximated by a second-order polynomial expansion $\tilde{\psi}_i(s) = \beta_0 + \beta_1 s + \beta_2 s^2$ using three interpolation points at \( s = 0 \), \( s = \Delta s_i \), and \( s = \Delta s_{-i} \). This polynomial approximation yields the following system
\begin{equation}\label{eqn:beta_coeff}
\left\{
\begin{aligned}
    \tilde{\psi}_i(0) &= V_n(\mathbf{x}_i), \\
    \tilde{\psi}_i(\Delta s_i) &= V_n(\mathbf{x}_{i+1}) \sqrt{1 + \big(p_i'(\Delta s_i)\big)^2}, \\
    \tilde{\psi}_i(\Delta s_{-i}) &= V_n(\mathbf{x}_{i-1}) \sqrt{1 + \big(p_i'(\Delta s_{-i})\big)^2},
\end{aligned}
\right.
\end{equation}
with an explicit solution:
\begin{equation}\label{eqn:beta_value}
\left\{
\begin{aligned}
    \beta_0 &= V_n(\mathbf{x}_i), \\
    \beta_1 &= V_n(\mathbf{x}_{i+1}) \frac{\Delta s_{-i} \sqrt{1 + \big(p_i'(\Delta s_i)\big)^2}}{\Delta s_i (\Delta s_{-i} - \Delta s_i)} 
            - V_n(\mathbf{x}_{i-1}) \frac{\Delta s_i \sqrt{1 + \big(p_i'(\Delta s_{-i})\big)^2}}{\Delta s_{-i} (\Delta s_{-i} - \Delta s_i)} 
            - V_n(\mathbf{x}_i) \frac{\Delta s_i + \Delta s_{-i}}{\Delta s_i \Delta s_{-i}}, \\
    \beta_2 &= V_n(\mathbf{x}_{i+1}) \frac{\sqrt{1 + \big(p_i'(\Delta s_i)\big)^2}}{\Delta s_i (\Delta s_i - \Delta s_{-i})} 
            - V_n(\mathbf{x}_{i-1}) \frac{\sqrt{1 + \big(p_i'(\Delta s_{-i})\big)^2}}{\Delta s_{-i} (\Delta s_i - \Delta s_{-i})} 
            + V_n(\mathbf{x}_i) \frac{1}{\Delta s_i \Delta s_{-i}}.
\end{aligned}
\right.
\end{equation}

Then the singular term $I_S$ can be approximated by,
\BEA
I_S \approx \tilde{I}_{S} = \underbrace{-\frac{1}{2\pi} \int_{\Delta s_{-i}}^{\Delta s_i}\beta_0\ln|s|ds}_{I_{S_1}} \underbrace{-\frac{1}{2\pi}\int_{\Delta s_{-i}}^{\Delta s_i}\beta_1 s\ln|s|ds}_{I_{S_2}}\underbrace{-\frac{1}{2\pi}\int_{\Delta s_{-i}}^{\Delta s_i}\beta_2 s^2\ln|s|ds}_{I_{S_3}},\label{eq:approxIs}
\EEA
where
\begin{equation}\label{eqn:singular}
    \begin{aligned}
        I_{S_1} &=-\frac{\beta_0}{2\pi}\Big(-\Delta s_{-i}\ln(-\Delta s_{-i})+\Delta s_{-i}+\Delta s_i\ln(\Delta s_i)-\Delta s_i\Big),\\
        I_{S_2}&= -\frac{\beta_1}{2\pi}\Big(-\frac{1}{2}(\Delta s_{-i})^2\ln(-\Delta s_{-i})+\frac{1}{4}(\Delta s_{-i})^2+\frac{1}{2}(\Delta s_i)^2\ln(\Delta s_i)-\frac{1}{4}(\Delta s_i)^2\Big),\\
        I_{S_3}&=-\frac{\beta_2}{2\pi}\Big(-\frac{1}{3}(\Delta s_{-i})^3\ln(-\Delta s_{-i})+\frac{1}{9}(\Delta s_{-i})^3+\frac{1}{3}(\Delta s_i)^3\ln(\Delta s_i)-\frac{1}{9}(\Delta s_i)^3\Big).
    \end{aligned}
\end{equation}
with $\beta_0, \beta_1,$ and $\beta_2$ given as \eqref{eqn:beta_value}. 
 
Combining \eqref{eqn:beta_value} and \eqref{eq:approxIs} for the singular term $\tilde{I}_S$ and applying the trapezoidal rule to obtain $\tilde{I}_N$ as a discrete approximation to the nonsingular term $I_N$ in \eqref{I_N} for simplicity, we obtain the matrix $\mathbf{A}$, whose components are reported in Appendix \ref{SM1} in the linear system $\mathbf A \mathbf{\tilde{V}}_n = \mathbf b$, which serves as a discrete approximation of $AV_n(\mathbf x)=b(\mathbf x)$ at $\mathbf{x}_i, i=1,\ldots, N$. Here, $\tilde{\mathbf V}_n$ is an approximation to $\mathbf V_n$, whose $i$th component is $(\mathbf{V}_n)_i = V_n(\mathbf x_i)$.


Next, we formulate $\mathbf b$ which is the discrete approximation of $b(\mathbf x)$ in \eqref{eq:b}, where $\nabla G(\mathbf x,\mathbf y) = \frac{1}{2\pi}\frac{\mathbf x-\mathbf y}{\|\mathbf x-\mathbf y\|^2}$. Similarly, we approximate $b(\mathbf x_i)$ with the local parameterized volume in \eqref{eq:approxJac} and write the integral in terms of the local coordinate system through a change of variables,
\begin{equation}\label{eqn:b_int}
\begin{aligned}
b(\mathbf x_i)\approx-\frac{\kappa(\mathbf x_i)}{2}&-\frac{1}{2}\Bigg[\int_{\Delta s_{-i}}^{\Delta s_i} \kappa\big(\iota_i(s)\big)\cdot\Big(\nabla G\big(\mathbf x_i,\iota_i(s)\big)\cdot\hat{\mathbf{n}}\big(\iota_i(s)\big)\Big)\sqrt{1+\big(p_i'(s)\big)^2}ds\\
&+\sum_{j\neq i}\int_{\Delta s_{-j}}^{\Delta s_j} \kappa\big(\iota_j(s)\big)\Big(\nabla G\big(\mathbf x_i,\iota_j(s)\big)\cdot \hat{\mathbf{n}}\big(\iota_j(s)
\big)\Big)\sqrt{1+\big(p_j'(s)\big)^2}ds\Bigg],
\end{aligned}
\end{equation}

which can be further discretized using any standard quadrature rule. Notice that the first integral has a singularity when $s=0$, because $\iota_i(0)=\mathbf x_i$. 

To handle the singularity, we calculate $\nabla G\big(\mathbf x_i,\iota_i(s)\big)\cdot\hat{\mathbf{n}}\big(\iota_i(s)\big)\big|_{s=0}$. We first derive the unit outer normal vector $\hat{\mathbf n}\big(\iota_i(s)\big)$ for the points in the neighborhood of $\mathbf x_i$. Since the unit tangent vector is given as,
\[\hat{\mathbf t}\big(\iota_i(s)\big)=\frac{\hat{\mathbf t}_i+\hat{\mathbf n}_ip_i'(s)}{\sqrt{1+\big(p_i'(s)\big)^2}}.\]
\comment{
the normal vector is 
\begin{equation*}
    \begin{aligned}
\frac{d\hat{\mathbf t}\big(\iota_i(s)\big)}{ds}&=\frac{\hat{\mathbf n}_ip''_i(s)\Big(1+(p'_i(s))^2\Big)-\Big(\hat{\mathbf t}_i+\hat{\mathbf n}_ip'_i(s)\Big)p'_i(s)p''_i(s)}{\Big(1+(p'_i(s))^2\Big)^\frac{3}{2}}=\frac{\hat{\mathbf n}_ip''_i(s)-\hat{\mathbf t}_ip'_i(s)p''_i(s)}{\Big(1+\big(p'_i(s)\big)^2\Big)^\frac{3}{2}}
    \end{aligned}
\end{equation*}
Therefore,} then the unit outer normal vector for the points in the neighborhood of $\mathbf x_i$ is 
\[\hat{\mathbf n}\big(\iota_i(s)\big)= \text{sign}\big(p''_i(s)\big) \frac{\frac{d\hat{\mathbf t}\big(\iota_i(s)\big)}{ds}}{\left\|\frac{d\hat{\mathbf t}\big(\iota_i(s)\big)}{ds}\right\|} =\text{sign}\big(p''_i(s)\big)\frac{\hat{\mathbf n}_ip''_i(s)-\hat{\mathbf t}_ip'_i(s)p''_i(s)}{\sqrt{\big(p_i''(s)\big)^2\Big(1+\big(p'_i(s)\big)^2\Big)}},\]
where $\text{sign}\big(p''_i(s)\big)$ ensures that the unit outer normal vector points outward from $\Omega(t)$.

Substitute $\hat{\mathbf n}\big(\iota_i(s)\big)$ into $\displaystyle\nabla G\big(\mathbf x_i,\iota_i(s)\big)\cdot\hat{\mathbf n}\big(\iota_i(s)\big)$, we have
\begin{equation*}
    \begin{aligned}
        \Big(\nabla G\big(\mathbf x_i,\iota_i(s)\big)\cdot\hat{\mathbf n}\big(\iota_i(s)\big)\Big)\Big\vert_{s=0}&=\frac{1}{2\pi}\frac{-s\hat{\mathbf t}_i-p_i(s)\hat{\mathbf n}_i}{s^2+p_i^2(s)}\cdot\hat{\mathbf n}\big(\iota_i(s)\big)\Bigg\vert_{s=0} = 
        \comment{\\
        &=\frac{\text{sign}\big(p''(s)\big)}{2\pi}\frac{\Big(-s\hat{\mathbf t}_i-p_i(s)\hat{\mathbf n}_i\Big)\cdot\Big(\hat{\mathbf n}_ip''_i(s)-\hat{\mathbf t}_ip'_i(s)p''_i(s)\Big)}{\Big(s^2+p_i^2(s)\Big)\sqrt{\big(p''_i(s)\big)^2\big(1+(p'_i(s))^2\big)}}\Big\vert_{s=0}\\
        &=\frac{\text{sign}\big(p''(s)\big)}{2\pi}\frac{sp'_i(s)p''_i(s)-p_i(s)p''_i(s)}{\Big(s^2+p_i^2(s)\Big)\sqrt{\big(p''_i(s)\big)^2\big(1+(p'_i(s))^2\big)}}\Bigg\vert_{s=0}\\
        &=\frac{\text{sign}\big(p''(s)\big)}{2\pi}\frac{p''_i(s)\Big(\frac{p'_i(s)}{s}-\frac{p_i(s)}{s^2}\Big)}{\Big(1+\frac{p_i^2(s)}{s^2}\Big)\sqrt{\big(p''_i(s)\big)^2\big(1+(p'_i(s))^2\big)}}\Big\vert_{s=0}\\
    &=\frac{\text{sign}(\alpha_{i,2})}{2\pi}\frac{2\alpha_{i,2}(2\alpha_{i,2}-\alpha_{i,2})}{|2\alpha_{i,2}|}\\
    &=}\frac{\alpha_{i,2}}{2\pi},
    \end{aligned}
\end{equation*}
where we have used the fact that $p_i(s) = \alpha_{i,2}s^2 + \ldots + \alpha_{i,\ell}s^\ell$ for $\ell \geq 2$.

For example, applying the trapezoidal rule to \eqref{eqn:b_int} for simplicity, we obtain the vector $\mathbf b$ as follows,

\begin{equation}\label{b}
\begin{aligned}
    \mathbf b_i =&-\frac{\kappa(\mathbf x_i)}{2}+\left(\kappa(\mathbf x_{i-1})\Big(\nabla G(\mathbf x_i,\mathbf x_{i-1})\cdot\hat{\mathbf n}_{i-1}\Big)\sqrt{1+\big(p'_i(\Delta s_{-i})\big)^2}+\kappa(\mathbf x_i)\frac{\alpha_{i,2} }{2\pi}\right)\frac{\Delta s_{-i}}{2}\notag \\ &-\left(\kappa(\mathbf x_i)\frac{\alpha_{i,2}}{2\pi}+\kappa(\mathbf x_{i+1})\Big(\nabla G(\mathbf x_i,\mathbf x_{i+1})\cdot\hat{\mathbf n}_{i+1}\Big)\sqrt{1+\big(p'_i(\Delta s_{i})\big)^2}\right)\frac{\Delta s_i}{2}\\
    &- \sum_{j\neq i-1, i}\left(\kappa(\mathbf x_j)\Big(\nabla G(\mathbf x_i,\mathbf x_j)\cdot\hat{\mathbf n}_j\Big)
    +\kappa(\mathbf x_{j+1})\Big(\nabla G(\mathbf x_i,\mathbf x_{j+1})\cdot\hat{\mathbf n}_{j+1}\Big)\sqrt{1+\big(p'_j(\Delta s_{j})\big)^2}\right)\frac{\Delta s_j}{2}.
\end{aligned}
\end{equation}

\subsection{Time discretization}
Now we apply two time discretization schemes to \eqref{velocity}, $\frac{d\mathbf x}{dt}=\tilde{\mathbf V}_n\hat{\mathbf n}$ with $\tilde{\mathbf V}_n=\mathbf A^{-1}\mathbf b$. For example, we will consider the standard forward Euler scheme and the second-order Runge-Kutta scheme. Denoting the time-discretization with a superscript $m$, that is $\mathbf{x}^m$ as an approximation to $\mathbf{x}(t_m)$ with time discretization $\Delta t = t_{m}-t_{m-1}$, we have:
\begin{scheme}[Forward Euler] \label{sch:FE}
\[ \mathbf x^{m+1} = \mathbf x^m + \Delta t (\tilde{\mathbf V}_n)^m \hat{\mathbf n}(\mathbf x^n).\]
\end{scheme}

\begin{scheme}[Second-order Runge-Kutta] \label{sch:rk2}
\[\mathbf x^{m+\frac{1}{2}} = \mathbf x^m+\frac{\Delta t}{2}(\tilde{\mathbf V}_n)^m\hat{\mathbf n}(\mathbf x^m),\,\, 
(\tilde{\mathbf V}_n)^{m+\frac{1}{2}}= \Big(\mathbf A^{-1}\Big)^{m+\frac{1}{2}}\mathbf b^{m+\frac{1}{2}}, \,\,
\mathbf x^{m+1}= \mathbf x^m+\Delta t (\tilde{\mathbf V}_n)^{m+\frac{1}{2}}\hat{\mathbf n}(\mathbf x^{m+\frac{1}{2}}).\]
\end{scheme}

\section{Error analysis}\label{sec4}
In this section, we provide an error analysis for the discrete approximation of the boundary integral $AV_n (\mathbf{x}) = b(\mathbf{x})$ on the data set $\{\mathbf{x}_i\}_{i=1,\ldots, N} \subset \Gamma(t)$ for a fixed time $t \geq 0$. In the following discussion, we will first deduce some technical results to be used for achieving the consistency of the discrete approximation of $AV_n(\mathbf{x}_i)$ and $b(\mathbf{x}_i)$.

\begin{prop}\label{prop4.1}
$X = \{\mathbf{x}_i\}_{i=1,\ldots, N}\subset \Gamma(t)$ are randomly sampled from a uniform distribution of a 1-dimensional smooth, $C^{\ell+1}$, boundary $\Gamma(t)$ (closed curve) at a fixed time $t\geq 0$. We denote $\mathbf{x}_{i-1}$ and $\mathbf{x}_{i+1}$ to be the two adjacent points to $\mathbf{x}_i$ on the boundary curve $\Gamma(t)$ for all $i=1,\ldots, N$, with periodic structure, that is, $\mathbf{x}_{N+1} = \mathbf{x}_1$ and $\mathbf{x}_0 = \mathbf{x}_N$. Suppose that we parameterize $\Gamma(t)$ locally at $\mathbf{x}_i$ with $p_i \in C^\ell(\mathbb{R})$ using Algorithm~\ref{alg:gmls}. Then, with probability higher than $1-\frac{2}{N}$,
\begin{eqnarray}
|\Delta s_i-\Delta s(\mathbf x_i)| =\mathcal{O}\left(\left(\frac{\log N}{N}\right)^{\ell+1}\right),
\end{eqnarray}
and with probability higher than $1-\frac{1}{N}$,
\begin{eqnarray}\label{ErrordS}
\left|dS_{\mathbf{y}} - \sqrt{1+\big(p_j'(s)\big)^2}ds \right | &=& \mathcal{O}\left(\left(\frac{\log N}{N}\right)^{\ell}\right),
\end{eqnarray}
as $N \to \infty$.
\end{prop}

\begin{proof}
Suppose that $X=\{\mathbf{x}_i\}_{i=1,\ldots, N}\subset \Gamma(t)$ are randomly sampled from a uniform distribution. We denote the $\mathbf{x}_{i-1}$ and $\mathbf{x}_{i+1}$ to be the two adjacent points to $\mathbf{x}_i$ on the boundary curve $\Gamma(t)$ for all $i=1,\ldots, N$, with periodic structure, that is, $\mathbf{x}_{N+1} = \mathbf{x}_1$ and $\mathbf{x}_0 = \mathbf{x}_N$. In such a configuration, it is easy to see that, 
\[
\max_{i}\frac{1}{2} d_g (\mathbf{x}_i,\mathbf{x}_{i+1}) = \sup_{\mathbf{x}\in \Gamma(t)} \min_{i} d_g(\mathbf{x},\mathbf{x}_i) = h_{X,\Gamma},
\]
which is known as the fill distance or sometimes regarded as the mesh size. For uniform random samples, it is known that (see Lemma~B.2 in \cite{harlim2023radial}) with probability higher than $1-\frac{1}{N}$, $h_{X,\Gamma} = \mathcal{O} \left(N^{-1}\log(N)\right)$ as $N\to \infty$. Let $h = d_g(\mathbf{x}_{i+1},\mathbf{x}_i)$, one can show that 
\[
\|\mathbf{x}_{i+1}- \mathbf{x}_i\| = h + \mathcal{O}(h^3),  
\]
as $N\to \infty$ (see Lemma 6 in \cite{cl:2006} or the proof of Proposition 3.1 in \cite{jiang2023ghost}). Together, we have that with probability higher than $1-\frac{1}{N}$,
\BEA
\max_i \|\mathbf{x}_{i+1}- \mathbf{x}_i\| = 2h_{X,\Gamma} + \mathcal{O}(h_{X,\Gamma}^3)= \mathcal{O}\left(\frac{\log N}{N}\right), \label{maxh}
\EEA
as $N\to\infty$.
Moreover, by Proposition~\ref{prop3.1}, the tangent vector estimate $\hat{\mathbf t}_i$ of the underlying tangent vector $\mathbf t(\mathbf x_i)$ converges at rate $(N^{-1}\log N)^{\ell}$, with probability higher than $1-\frac{2}{N}$,
\[ |\Delta s_i-\Delta s(\mathbf x_i)|=\left| \hat{\mathbf t}_i^\top(\mathbf x_{i+1} - \mathbf x_i)-\mathbf t(\mathbf x_i)^\top(\mathbf x_{i+1} - \mathbf x_i)\right|\leq \|\hat{\mathbf t}_i-\mathbf t(\mathbf x_i)\| \|\mathbf x_{i+1} - \mathbf x_i\|=\mathcal{O}\left(\left(\frac{\log N}{N} \right)^{\ell+1}\right),\]
as $N\to \infty$.

Also, from $dS_{\mathbf y} = \sqrt{|\iota'(s)|}ds \approx  \sqrt{1+\big(p_j'(s)\big)^2}ds$, one can verify that
\[\left|\sqrt{\hat{\mathbf t}_i^\top\hat{\mathbf t}_i}-\sqrt{\mathbf t(\mathbf x_i)^\top\mathbf t(\mathbf x_i)}\right|=\left|\frac{\hat{\mathbf t}_i^\top\hat{\mathbf t}_i-\mathbf t(\mathbf x_i)^\top\mathbf t(\mathbf x_i)}{\sqrt{\hat{\mathbf t}_i^\top\hat{\mathbf t}_i}+\sqrt{\mathbf t(\mathbf x_i)^\top\mathbf t(\mathbf x_i)}}\right|
=\mathcal{O}\left(\left(\frac{\log N}{N} \right)^{\ell}\right).\]
This completes the proof.
\end{proof}

Now we study the consistency of our  approximation to $AV_n(\mathbf x)=\int_\Gamma G(\mathbf x,\mathbf y)V_n(\mathbf y)dS_{\mathbf y}$ at each $\mathbf{x}_i$. To make the discussion concise, we define $AV_n(\mathbf{x}_i)= I_{S}^* + I_{N}^*$,
where
\BEA
I_{S}^*&=&-\frac{1}{2\pi}\int_{\Delta s(\mathbf x_{-i})}^{\Delta s(\mathbf x_i)} \ln |s|V_n(\iota(s))\sqrt{|\iota'(s)|}ds,\notag\\
I_{N}^* &=& -\frac{1}{4\pi}\int_{\Delta s(\mathbf x_{-i})}^{\Delta s(\mathbf x_i)} \ln\Big(1+\frac{p^2(s)}{s^2}\Big)V_n\big(\iota(s)\big)\sqrt{|\iota'(s)|}ds +\sum_{j\neq i-1,i}\int_0^{\Delta s(\mathbf{x}_j)} G\big(\mathbf x_i,\iota(s)\big)V_n\big(\iota(s)\big)\sqrt{|\iota'(s)|}ds,\notag
\EEA
with $p(s)$ representing the true parameterization of the boundary curve. The proposition below follows from estimating the errors in approximating these integrals by $\tilde{I}_S$, the approximation to $I_S$ in \eqref{eq:approxIs}, and $\tilde{I}_N$, the quadrature (e.g., Newton's Cotes family) approximation to $I_N$ in \eqref{I_N}.  

\begin{prop}\label{prop4.2}
Assume that the assumptions of Proposition~\ref{prop3.1} and Proposition~\ref{prop4.1} hold. In addition, we also assume that $V_n \in C^3(\Gamma)$. Then, with probability higher than than $1-\frac{23}{N}$,
\[
\left|AV_n(\mathbf{x}_i) - (\mathbf{AV}_n)_i \right|\leq |I_S^*-\tilde{I}_S| + |I_N^*-\tilde{I}_N|=\mathcal{O}\left(\frac{\big(\log N\big)^{\ell+2}}{N^{\ell-2}}\right) + \mathcal{O}\left(\left(\frac{\log N}{N}\right)^{q}\right),
\]
as $N\to\infty$. Here, we denote $\tilde{I}_N$ to be the quadrature approximation to $I_N$ in \eqref{I_N} of order $q$, and $\tilde{I}_S$ as the approximation in \eqref{eq:approxIs} to $I_S$ in \eqref{eq_sing}.
\end{prop}

\begin{proof} In the following, we bound the error in estimating $I_S^*$. We report the bound for estimating the error for $I_N^*$ in Appendix \ref{SM2} since it follows similar argument.
Note that
\begin{eqnarray}
|I_S^* - \tilde{I}_S| &\leq& |I_S^* - I_S | + |I_S - \tilde{I}_S| \notag =
\frac{1}{2\pi}
\left|\int_{\Delta s(\mathbf x_{-i})}^{\Delta s(\mathbf x_{i})}\ln|s|\psi^*(s)ds-\int_{\Delta s_{-i}}^{\Delta s_i}\ln|s|\psi_i(s)ds\right| \notag \\ &&+ \frac{1}{2\pi}\left\vert \int_{\Delta s_{-i}}^{\Delta s_i}\ln|s|\psi_i(s)ds-\int_{\Delta s_{-i}}^{\Delta s_i}\ln|s|\tilde{\psi}_i(s)ds\right\vert\label{ErrorIS}
\end{eqnarray}
where $\tilde{\psi}_i = \beta_0 + \beta_1s+ \beta_2 s^2$ denotes the second-order polynomial approximation to $\psi_i(s):=V_n\big(\iota_i(s)\big)  \sqrt{1 + \big(p_i'(s)\big)^2}$ as discussed in \eqref{eqn:beta_coeff} and we have defined $\psi^*(s):= V_n\big(\iota(s)\big)\sqrt{|\iota'(s)|}$ to simplify the notation. 

By \eqref{ErrordS}, Proposition~\ref{prop3.1} and Proposition~\ref{prop4.1}, with probability higher than $1-\frac{2}{N}$,
\BEA
\left|\psi^*(s) - \psi_i(s)\right| &\leq& \left| V_n\big(\iota(s)\big)\right| \left|\sqrt{|\iota'(s)|} - \sqrt{1+\big(p_i'(s)\big)^2}\right|  + \sqrt{1+\big(p_i'(s)\big)^2}\left| V_n\big(\iota(s)\big) - V_n\big(\iota_i(s)\big)  \right|\notag \\
&\leq& \mathcal{O}\left(\left(\frac{\log N}{N}\right)^\ell\right) + C |\iota(s) - \iota_i(s)| = \mathcal{O}\left(\left(\frac{\log N}{N}\right)^\ell\right) + \mathcal{O}\left(\left(\frac{\log N}{N}\right)^{\ell+1}\right) 
\label{Errorpsi}
\EEA

Then the first error term in \eqref{ErrorIS} can be bounded as follows,
\begin{eqnarray}
&&\left|\int_{\Delta s(\mathbf x_{-i})}^{\Delta s(\mathbf x_{i})}\ln|s|\psi^*(s)ds-\int_{\Delta s_{-i}}^{\Delta s_i}\ln|s|\psi_i(s)ds\right| \notag \\
&&\leq \left|\int_{\Delta s(\mathbf x_{-i})}^{\Delta s(\mathbf x_{i})}\ln|s|\psi^*(s)ds - \int_{\Delta s_{-i}}^{\Delta s_i}\ln|s|\psi^*(s)ds\right| + \left|\int_{\Delta s_{-i}}^{\Delta s_i}\ln|s|\big(\psi^*(s)-\psi_i(s)\big)ds\right|\notag \\ 
&&\leq  C \big(|\Delta s_i - \Delta s(\mathbf{x}_i)| + |\Delta s_{-i} - \Delta s(\mathbf{x}_{-i})|\big) + \max_{s\in [\Delta s_{-i},\Delta s_{i}]}\left|\psi^*(s)-\psi_i(s)\right| \left|\int_{\Delta s_{-i}}^{\Delta s_i}\ln|s|ds\right|\notag \\
&&= \mathcal{O}\left(\left(\frac{\log N}{N}\right)^{\ell+1}\right) + \mathcal{O}\left(\left(\frac{\log N}{N}\right)^{\ell+1}\ln\left(\frac{\log N}{N}\right)\right)=\mathcal{O}\left(\left(\frac{\log N}{N}\right)^{\ell+1}\ln\left(\frac{\log N}{N}\right)\right),\notag
\end{eqnarray}
with probability higher than $1-\frac{7}{N}$, where we used the bounds from Proposition~\ref{prop4.1}. 

The second error term in \eqref{ErrorIS} is bounded by the quadratic interpolation error of the integrand,
\BEA
\left\vert \int_{\Delta s_{-i}}^{\Delta s_i}\ln|s|\psi_i(s)ds-\int_{\Delta s_{-i}}^{\Delta s_i}\ln|s|\tilde{\psi}_i(s)ds\right\vert &\leq&\max_{s\in [\Delta s_{-i},\Delta s_{i}]}\left|\psi_i(s)-\tilde{\psi}_i(s)\right|\left| \int_{\Delta s_{-i}}^{\Delta s_i}\ln|s| ds\right| \notag \\ 
&\leq &
C  h_{X,\Gamma}^3 \big(\Delta s_{i}\ln(\Delta s_i)-\Delta s_{-i}\ln(-\Delta s_{-i})\big)= \mathcal{O}\left(\left(\frac{\log N}{N}\right)^4\ln\left(\frac{\log N}{N}\right)\right),\notag
\EEA
with probability higher than $1-\frac{2}{N}$ for some constant $C$, under the assumption that $\psi_i$ is $C^3$.

Asserting the rates in these bounds to \eqref{ErrorIS} and accounting for all probabilities, we obtain
\BEA
|I_S^* - \tilde{I}_S | = \mathcal{O}\left(\left(\frac{\log N}{N}\right)^{\ell+1}\ln\left(\frac{\log N}{N}\right)\right)+\mathcal{O}\left(\left(\frac{\log N}{N}\right)^4\ln\left(\frac{\log N}{N}\right)\right)\label{ErrorIs}
\EEA
with probability higher than $1-\frac{9}{N}$. 
\end{proof}

Next, we report the consistency of our approximation to \eqref{eq:b}. See Appendix \ref{SM3} for the detailed proof, which follows the same strategy as the proof above.

\comment{Move this to an Appendix: For simplicity, we define
\[
b(\mathbf{x}_i)= -\frac{\kappa(\mathbf x_i)}{2} + \frac{1}{2}I_{b}^*, 
\]
where
\[I^*_b=-\sum_{j}\int_{\Delta s(\mathbf x_{-j})}^{\Delta s(\mathbf x_j)} \kappa\big(\iota(s)\big)\Big(\nabla G\big(\mathbf x_i,\iota(s)\big)\cdot \mathbf{n}\big(\iota(s)
\big)\Big)\sqrt{|\iota'(s)|}ds.\]
}

\begin{prop}\label{prop4.3}
Let the assumptions of Proposition~\ref{prop3.1} and Proposition~\ref{prop4.1} be valid.  Then, with probability higher than than $1-\frac{10}{N}$,
\[
\left|b(\mathbf{x}_i) - \mathbf b_i \right| \comment{\leq\frac{1}{2}\left|\kappa(\mathbf x_i)-\kappa_i\right|+\frac{1}{2}|I_b^*-\tilde{I}_b|}=\mathcal{O}\left(\frac{\big(\log N\big)^{\ell+2}}{N^{\ell-2}}\right)+\mathcal{O}\left(\left(\frac{\log N}{N}\right)^{q}\right),
\]
as $N\to\infty$, where $q$ denotes the quadrature rule error rate. 
\comment{Here, we denote $\tilde{I}_b$ as the Simpson's approximation to $I_b$ with
\[I_b=\sum_{j}\int_{\Delta s_{-j}}^{\Delta s_j} \kappa\big(\iota_j(s)\big)\Big(\nabla G\big(\mathbf x_i,\iota_j(s)\big)\cdot \hat{\mathbf{n}}\big(\iota_j(s)
\big)\Big)\sqrt{1+\big(p_j'(s)\big)^2}ds.\]}
\end{prop}

In the following, we verify the invertibility (and stability) of the estimator $\mathbf{A}$ of the linear system $\mathbf{A}\mathbf{V}_n = \mathbf{b}$ that will be numerically solved. In this context, recall that $(\mathbf{V}_n)_i=V_n(\mathbf{x}_i)$. Our approach will be based on the following result.

\begin{prop}[McLean \cite{mclean1991variation}.]
The operator $A$ as defined in \eqref{eq:AVn} is self-adjoint and coercive under the inner product induced by the Hilbert space, 
\[
H_* = \left\{ \nu \in L^2(\Gamma) \;|\; \langle \nu, 1 \rangle = \int_\Gamma \nu(\mathbf{x})\,dS_{\mathbf{x}} = 0 \right\}.
\]
The coercivity can be expressed as follows. For all function $\nu \in H_*$,
\[
\langle \nu, A\nu \rangle = \int_{\Gamma} \nu(\mathbf{x})A\nu(\mathbf{x}) dS_{\mathbf{x}} \geq \zeta \|\nu\|^2_{L^2(\Gamma)},
\]
where $\zeta >0$ denotes the coercivity constant.
\end{prop}

With this result, we now state the invertibility and positivity condition for the matrix $\mathbf{A}$. 

\begin{prop}\label{prop4.5} Let the assumptions in Proposition~\ref{prop4.2} to be valid. With probability higher than $1-\frac{25}{N}$, the matrix $\mathbf A$ is invertible when the coercivity constant of $A$, 
\[
\zeta> C^{-1}(\log N)^{-1}\bar{\delta}(N), 
\]
for some constant $C>0$, and,
\[
\bar{\delta}(N) = \mathcal{O}\left(\frac{\big(\log N\big)^{\ell+2}}{N^{\ell-\frac{3}{2}}}\right) + \mathcal{O}\left(\left(\frac{\log(N)}{N}\right)^{\frac{1}{2}}\right)
\]
which decays to zero as $N\to \infty$.
\end{prop}
\begin{proof}
Let $\{\eta_i(\mathbf{x})\}_{i=1}^n$ be linear hat functions defined on a parameterization $\iota: I \subset \mathbb{R} \to \Gamma$ such that $\iota(s) = \mathbf{x}$. That is, for all $\mathbf{x}_i \in \Gamma$,
\[
\eta_i (\mathbf{x}) = \eta_i \big(\iota(s)\big) = \begin{cases}\frac{s-s_{i-1}}{s_i-s_{i-1}}, & \mbox{for } s \in \left[\iota^{-1}(\mathbf{x}_{i-1}),\iota^{-1}(\mathbf{x}_{i})\right] \\ \frac{s_{i+1}-s}{s_{i+1}-s_i}, & \mbox{for } s \in \left[\iota^{-1}(\mathbf{x}_{i}),\iota^{-1}(\mathbf{x}_{i+1})\right]  \end{cases}. 
\]
For any $V_n \in L^2(\Gamma)$, we can write
\[
V_n(\mathbf x) = \sum_{i=1}^N  V_n(\mathbf{x}_i) \eta_i(\mathbf x).
\]
To be in $H^*$, we require that $V_n(\mathbf{x}_i)$ satisfies,
\[
\sum_{i=1}^N V_n(\mathbf{x}_i) m_i = 0, \quad m_i = \int_\Gamma \eta_i(\mathbf{x})\, dS_{\mathbf{x}},
\]
which is obtained by imposing the \emph{discrete zero-mean constraint}. Then,
\BEA
\zeta \langle V_n, V_n \rangle = \zeta \sum_{i,j=1,\ldots, N} V_n(\mathbf{x}_i) V_n(\mathbf{x}_j) \int_{\Gamma} \eta_i(\mathbf{x})\eta_j(\mathbf{x})dS_{\mathbf{x}} = \zeta \mathbf{V}_n^\top \mathbf{M} \mathbf{V}_n \geq \zeta \lambda_{min}(\mathbf{M}) \|\mathbf{V}_n\|^2,\label{lowerbound}
\EEA
where
\[
\mathbf{M}_{ij} = \int_\Gamma \eta_i(\mathbf{x})\eta_j(\mathbf{x}) dS_{\mathbf{x}},
\]
which is SPD. Letting $\iota(s)= \mathbf{x}_0 + \mathbf t(\mathbf{x}_0) s + \mathbf n(\mathbf{x}_0) p(s)$, denoting a parameterization of $\Gamma$ at $\mathbf{x}_0$, one can write
\BEA
\mathbf{M}_{ij} &=& \int_{\iota^{-1}(\Gamma)}\eta_i\big(\iota(s)\big)\eta_j\big(\iota(s)\big) \sqrt{1+ \big(p'(s)\big)^2} \chi_{\{\eta_i(\Gamma) \cap \eta_j(\Gamma) \neq 0\} } (s)ds  \notag\geq \mathbf{L}_{ij} \notag \\ &=&  \int_{\iota^{-1}\big(\chi_{\{\eta_i(\Gamma) \cap \eta_j(\Gamma) \neq 0\} }\big)}\eta_i\big(\iota(s)\big)\eta_j\big(\iota(s)\big)ds \geq  \int_{-q_{X,\Gamma}}^{q_{X,\Gamma}} \eta_i\big(\iota(s)\big)\eta_j\big(\iota(s)\big)ds = \begin{cases} \frac{2}{3}q_{X,\Gamma} & \mbox{if }i=j \\  \frac{1}{4}q_{X,\Gamma}  & \mbox{if } |i-j| = 1, \\ 0, & \mbox{otherwise,} \end{cases} \notag
\EEA
where $q_{X,\Gamma} = \min_{i,j} d_g(\mathbf{x}_i,\mathbf{x}_j)$ denotes the separation distance. It is easy to check that $\mathbf{M}-\mathbf{L}$ is symmetric positive semi-definite. Thus, 
\BEA
\lambda_{min}(\mathbf{M}) \geq \lambda_{min}(\mathbf{L}) = \frac{2}{3}q_{X,\Gamma} + \frac{2}{4} q_{X,\Gamma} \cos\left(\frac{N\pi}{N+1}\right) \geq \frac{q_{X,\Gamma}}{6},\label{lambdamin}
\EEA
with probability higher than $1-\frac{1}{N}$, where we used Theorem~2.2 in \cite{kulkarni1999eigenvalues}.

 Since $V_n$ is continuous and bounded by the assumption in Proposition~\ref{prop4.2}, it is clear that $AV_n$ is also bounded (see \cite{kress1999linear}). This implies that $|V_n(x) AV_n(x)| \leq \tilde{C} < \infty$. By Hoeffding's inequality, with probability higher than $1-\frac{1}{N}$,
\[
\langle \eta_j,A\eta_i \rangle  - \frac{1}{N} \sum_{k=1}^N \eta_j(\mathbf{x}_k)A\eta_i(\mathbf{x}_k) \leq  \tilde{C}\left(\frac{\log N}{N}\right)^{1/2}, 
\]
for i.i.d. samples $X= \{\mathbf{x}_i\}_{i=1,\ldots, N} \subset \Gamma$.
By Proposition~\ref{prop4.2}, with probability higher than $1-\frac{24}{N}$
\BEA
\langle \eta_j,A\eta_i \rangle  &\leq& \frac{1}{N} \sum_{k=1}^N \eta_j(\mathbf{x}_k)A\eta_i(\mathbf{x}_k) + \tilde{C}\left(\frac{\log N}{N}\right)^{1/2} \notag \\ &\leq& \frac{1}{N} \sum_{k=1}^N \eta_j(\mathbf{x}_k)(\mathbf{A\eta}_i)_k + \underbrace{\frac{1}{N}\sum_{k=1}^N \eta_j(\mathbf{x}_k)\epsilon(N)+ \tilde{C} \left(\frac{\log N}{N}\right)^{1/2}}_{=\delta(N)}, \notag 
\EEA
where we have denoted the upper bound in Proposition~\ref{prop4.2} as, $A\eta_i(\mathbf{x}_k) - (\mathbf{A}\boldsymbol{\eta}_i)_k  \leq \epsilon(N)$. Thus,
\[
\delta(N) \leq \bar{\delta}(N) =  \frac{1}{\sqrt N}\|\boldsymbol{\eta}\|^2 \epsilon(N)  + \tilde{C} \left(\frac{\log N}{N}\right)^{1/2} =  \mathcal{O}\left(\frac{\big(\log N\big)^{\ell+2}}{N^{\ell-\frac{3}{2}}}\right) + \mathcal{O}\left(\frac{\log(N)}{N}\right)^{\min\{\frac{1}{2},q+\frac{1}{2}\}}.
\]
This implies that,
\BEA
\langle V_N,AV_N \rangle &=& \sum_{i,j = 1}^N V_n(\mathbf{x}_i)V_n(\mathbf{x}_j) \langle \eta_j,A\eta_i \rangle  \notag \\ &\leq& \sum_{i,j = 1}^N \frac{1}{N} \sum_{k=1}^N V_n(\mathbf{x}_j) \eta_j(\mathbf{x}_k)(\mathbf{A\eta}_i)_k V_n(\mathbf{x}_i) + \sum_{i,j = 1}^N V_n(\mathbf{x}_i)V_n(\mathbf{x}_j) \bar{\delta}(N) \notag \\
&=& \frac{1}{N}\sum_{k=1}^N V_n(\mathbf{x}_k)(\mathbf{AV}_n)_k + \sum_{i,j = 1}^N V_n(\mathbf{x}_i)V_n(\mathbf{x}_j) \bar{\delta}(N)\leq\frac{1}{N}\sum_{k=1}^N V_n(\mathbf{x}_k)(\mathbf{AV}_n)_k + N \|\mathbf{V}_n\|^2\bar{\delta}(N),\notag\label{bilinear}
\EEA
where we applied the Cauchy-Schwartz inequality. Combining \eqref{lowerbound}, \eqref{lambdamin}, and \eqref{bilinear},
\BEA
\frac{1}{N}\mathbf{V}_n^\top \mathbf{A}\mathbf{V}_n \geq \big(\zeta Cq_{X,\Gamma} - N\bar{\delta}(N)\big)\|\mathbf{V}_n\|^2,\label{discrete_coercive}
\EEA
which is positive definite when $\zeta > C^{-1}Nq_{X,\Gamma}^{-1}\bar{\delta}(N)$. By \eqref{maxh}, with probability higher than $1-\frac{1}{N}$, $q_{X,\Gamma}\leq h_{X,\Gamma} \leq N^{-1}\log(N)$, then it is clear that, $\zeta > C^{-1}(\log N)^{-1}\bar{\delta}(N)$.
Therefore, the matrix $\mathbf{A}$ is positive and it is, therefore, invertible. The statement of the theorem is complete by counting the probability.
\end{proof}

With Propositions~\ref{prop4.2}, \ref{prop4.3} and \ref{prop4.5}, we conclude the following convergence result for the system $\mathbf A\mathbf {\tilde{V}}_n=\mathbf b$. Particularly, we deduce the error of $\tilde{\mathbf{V}}_n$ as a discrete approximation to $\mathbf{V}_n$, whose components are defined as $(\mathbf{V}_n)_i = V_n(\mathbf{x}_i)$ that satisfy $AV_n(\mathbf{x}_i)=b(\mathbf{x}_i)$ for all $i=1,\ldots, N$.

\begin{theo}
Let the assumptions in Propositions~\ref{prop4.2}, \ref{prop4.3}, and \ref{prop4.5} be valid. Then with probability higher than $1-\frac{58}{N}$,
\[
\left\|\mathbf{V}_n - \mathbf{\tilde{V}}_n\right\|  = \mathcal{O}\left(\frac{\big(\log N\big)^{\ell+2}}{N^{\ell-\frac{3}{2}}}\right)+\mathcal{O}\left(\left(\frac{(\log N)^q}{N^{q+\frac{1}{2}}}\right)\right),
\]
for all $\mathbf{x}_i \in X \subset \Gamma$, as $N\to \infty$.
\end{theo}

\begin{proof}
Define $\bar{\zeta} = \zeta C h - N \bar{\delta}(N) > 0$. Let $\mathbf{E} := \mathbf{V}_n - \mathbf{\tilde{V}}_n$, then with probability higher than $1-\frac{25}{N}$,
\[
\left\|\mathbf{E}\right\|^2 \leq \frac{1}{\tilde{\zeta}N} \mathbf{E}^\top \mathbf{AE} \leq \frac{1}{\tilde{\zeta}N}\|\mathbf{E}\| \|\mathbf{AE}\|.
\]
We note that for any $i=1,\ldots,N$, with probability higher than $1-\frac{58}{N}$,
\BEA
|(\mathbf{AE})_{i}| &=& \left|(\mathbf{AV}_n)_{i} -  (\mathbf{A\tilde{V}}_n)_{i} - (AV_n)(\mathbf{x}_i) + (AV_n)(\mathbf{x}_i)\right| \leq   \left|(\mathbf{AV}_n)_{i} -(AV_n)(\mathbf{x}_i)\right| + \left|b(\mathbf{x}_i) - \mathbf{b}_{i}\right| \notag \\
&=& \mathcal{O}\left(\frac{\big(\log N\big)^{\ell+2}}{N^{\ell-2}}\right)+\mathcal{O}\left(\left(\frac{\log N}{N}\right)^{q}\right),
\EEA
as $N\to \infty$. So,
\[
\|\mathbf{E}\| \leq \frac{1}{\tilde{\zeta}N}\|\mathbf{AE}\| = \frac{1}{\tilde{\zeta} N^{1/2}} \left(\mathcal{O}\left(\frac{\big(\log N\big)^{\ell+2}}{N^{\ell-2}}\right)+\mathcal{O}\left(\left(\frac{\log N}{N}\right)^{q}\right)\right),
\]
which vanishes as $N\to \infty$.
\end{proof}

\section{Numerical Experiments}\label{sec5}
In this section, we present numerical experiments for \eqref{BIM} on a variety of smooth closed curves to illustrate the effectiveness of the proposed algorithms. Convergence studies are carried out for the circular case, where the exact solution is available. For simplicity, we use $k=\sqrt{N}$ (number of KNN to be used, chosen to be an odd number) where $N$ is the number of sampling points and degree $\ell = 6$ polynomial with uniformly sampled initial data points in the GMLS approximation for all examples.

\subsection{Free boundary equation on circle}
First, we consider the free boundary equation with a prescribed time-dependent forcing term, $h(t)$,
\begin{equation}\label{eq:FBP_circle}
\left\{
\begin{array}{rcll}
-\Delta p &=& 0& \text{in }\, \Omega(t), \\
p &=& \kappa & \text{on }\, \Gamma(t), \\
\frac{\partial p}{\partial \mathbf n} &=& -V_n +h(t)& \text{on }\, \Gamma(t),
\end{array}
\right.
\end{equation}
where $h(t) = 500\cos(500\pi t)$.

Then, integrating the equations on the free boundary leads to the following system,
\begin{equation}\label{eqn:BIE_circle}
\left\{
\begin{aligned}
&\int_{\Gamma}G(\mathbf{x},\mathbf{y})\Big(V_n(\mathbf y) -h(t)\Big)dS_\mathbf{y}=-\frac{\kappa(\mathbf{x})}{2} -\int_\Gamma \kappa(\mathbf{y})\frac{\partial G(\mathbf{x},\mathbf{y})}{\partial \mathbf{n}(\mathbf y)}dS_\mathbf{y}  ,\quad \mathbf{x},\mathbf{y} \in \Gamma(t),\\
&\frac{d\mathbf x}{dt} = V_n(\mathbf x)\mathbf n(\mathbf x), \quad \mathbf x\in\Gamma(t),
\end{aligned}
\right.
\end{equation}

With the initial condition given by the unit circle, 
\begin{equation}\label{eq:IC1}
\mathbf x(0)=\Big(x_1(0),x_2(0)\Big)=\Big(\cos(\theta),\sin(\theta)\Big),\quad \theta\in[0,2\pi],
\end{equation}
The solution at any time $t$ is given by
\begin{equation*}
\begin{aligned}
    R(t) &= 1 + \frac{1}{\pi}\sin(500\pi t),\\
    \mathbf x(t)&=\Big(x_1(t),x_2(t)\Big)=\Big(R(t)\cos(\theta),R(t)\sin(\theta)\Big),\quad\theta\in[0,2\pi],
    \end{aligned}
\end{equation*}
where $R(t)$ represents the radius of circle. Then we conduct the mesh refinement tests to check the order of spatial convergence and temporal convergence. First, we apply our method to discretize \eqref{eqn:BIE_circle} directly on point clouds with $N=100$, $200$, $400$, $800$, and use forward Euler (\ref{sch:FE}) with time step $\Delta t=10^{-5}$. We quantify the numerical error in approximating $V_n(\mathbf x)$ with, \[\|e_V\|_{L^2}=\sqrt{\frac{1}{|\Gamma|}\int_\Gamma \|\tilde{\mathbf V}_n(\mathbf x)- V_n(\mathbf x)\|^2dS_{\mathbf x}}\approx\sqrt{\frac{1}{N}\sum_{i=1}^N \Big((\tilde{\mathbf V}_n)_i- V_n(\mathbf x_i)\Big)^2},\] 
at the time $t=10^{-3}$. In Figure \ref{fig:space-convergence}, we show the convergence results corresponding to two numerical quadrature rules, the trapezoidal and Simpson's rule, applied in approximation of $b$ (see \eqref{eqn:b_int}). As for the approximation of $A$, we consider the trapezoidal rule (with the discrete component $\mathbf{A}_{ij}$ as reported above \eqref{eqn:b_int}). For this particular example, notice that the error rate of Simpson's rule is better than that of the trapezoidal despite the fact that the operator $A$ is discretized with the same trapezoidal rule. The main reason here is that in this circular case, the first term in Eqn.~\eqref{eqn:BIE_circle} can be written as $AV_n(\mathbf{x}) = b(\mathbf{x}) + h(t)A\mathbf{1}(\mathbf{x})$, where $b(\mathbf x) = 0$ for the circle and $V_n(\mathbf{x}) = h(t)\mathbf{1}(\mathbf x)$. Denoting the approximation of $\mathbf{A} = A + \delta A$ and $\mathbf{b} = b(\mathbf x) + \delta b = \delta b$, and $\mathbf{V}_n = V_n(\mathbf x) + \delta V_n$, where the ``delta" terms denote the corresponding numerical errors, one can deduce that $(A+\delta A) \delta V_n = \delta b$. Applying the Woodbury identity, we obtain $\delta V_n=(I+A^{-1}\delta A)^{-1}A^{-1}\delta b$, which explains why the error rate is dominated by the error induced by the quadrature rule in the estimation of $b$ as observed in Figure~\ref{fig:space-convergence}. 

Next, we conduct the time convergence studies with $N=400$ and calculate the numerical solutions to $t=4\times 10^{-3}$ with various time steps. The max-norm relative errors $\|e_R\|_{\ell^\infty}$ for radius of the circle using forward Euler (\ref{sch:FE}) and second-order Runge Kutta (\ref{sch:rk2}) are reported in Figure \ref{fig:time-convergence}. It indicates the first-order accuracy in time with the forward Euler method and second-order accuracy in time with the second-order Runge-Kutta method. Also, the evolution dynamics with $N=400$, $\Delta t=10^{-5}$, and forward Euler method are shown in Figure \ref{fig:boundary-dynamics} which depicts the boundary at various times, $t=0$, $0.001$, and $0.003$. The absolute error in the radius over time is shown in Figure \ref{fig:radius-error}, where the errors exhibit oscillations with a period of approximately $4\times 10^{-3}$.

\begin{figure}[htbp]
\center
\subfigure[Spatial mesh refinement test for $V_n$.]{\includegraphics[width=0.45\linewidth]{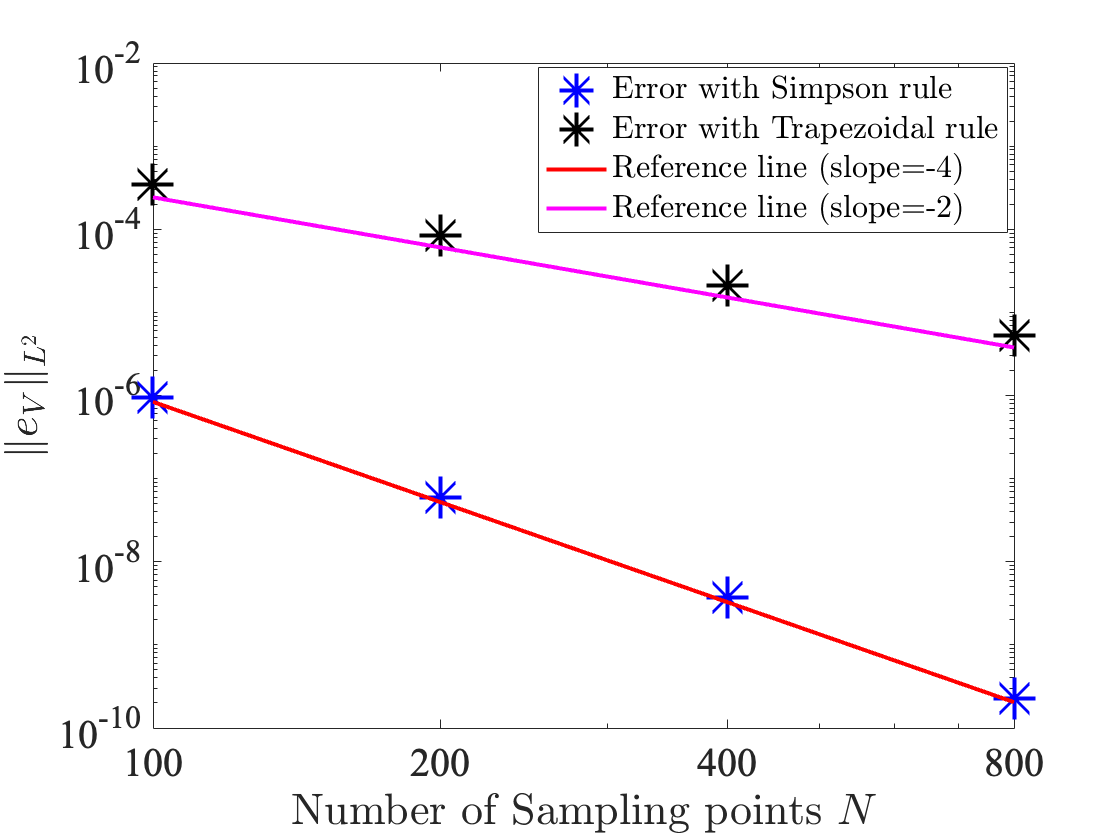}
\label{fig:space-convergence}}
\subfigure[Temporal mesh refinement test for radius of the circle.]{\includegraphics[width=0.45\textwidth]{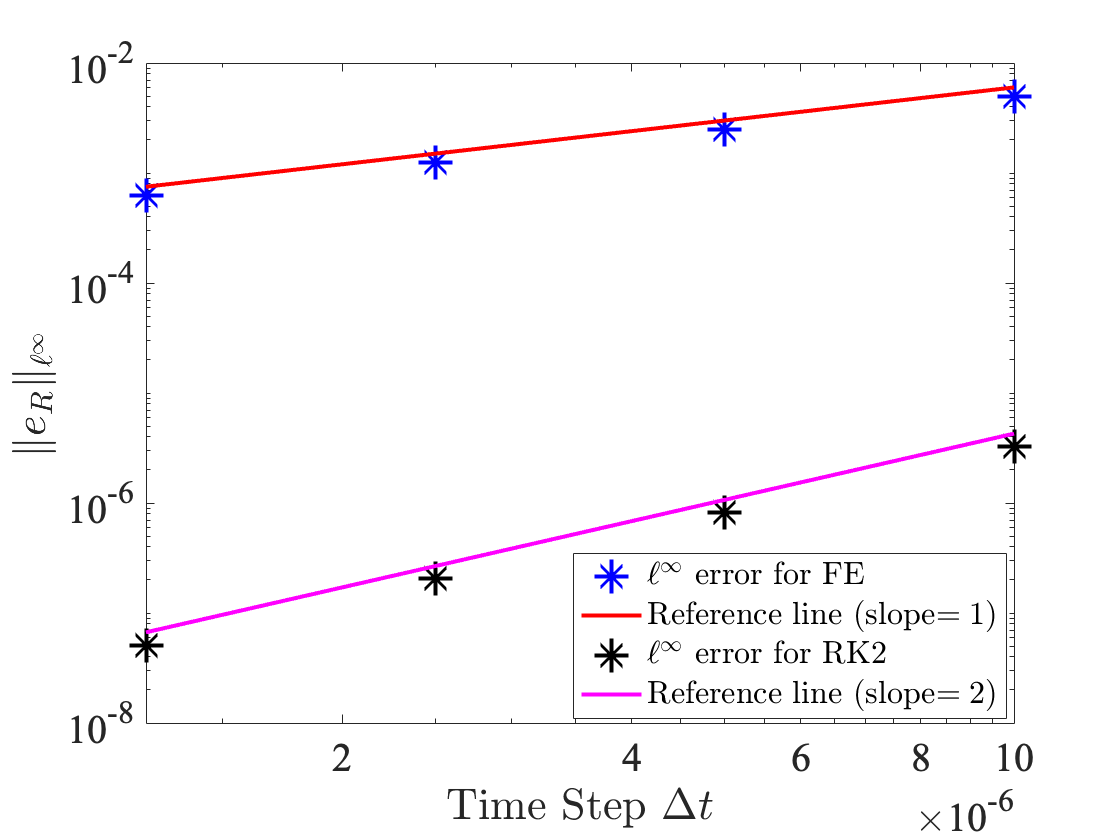}
\label{fig:time-convergence}}
\caption{Mesh refinement tests.}
\label{fig:convergence-test}
\end{figure}

\begin{figure}[htbp]
\center
\subfigure[Evolution of the circle boundary over time.]{\includegraphics[width=0.45\textwidth]{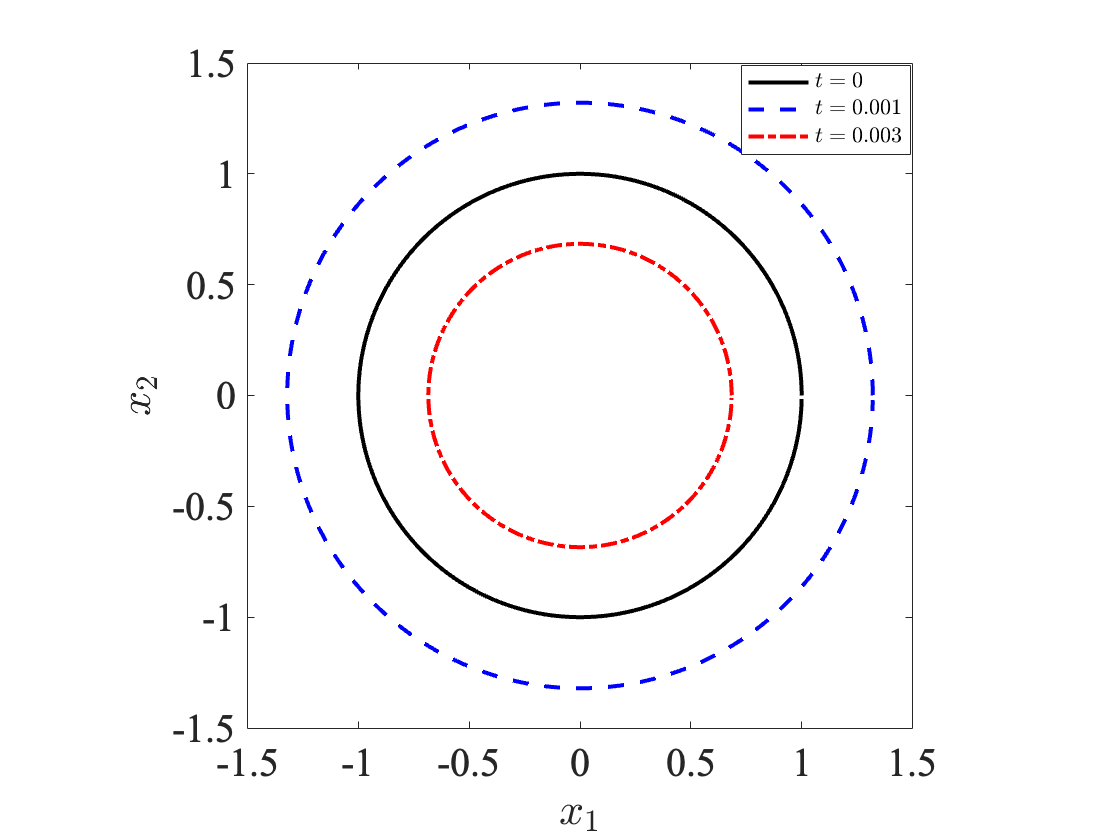}\label{fig:boundary-dynamics}}
\subfigure[Error for the circle radius over time.]{\includegraphics[width=0.45\textwidth]{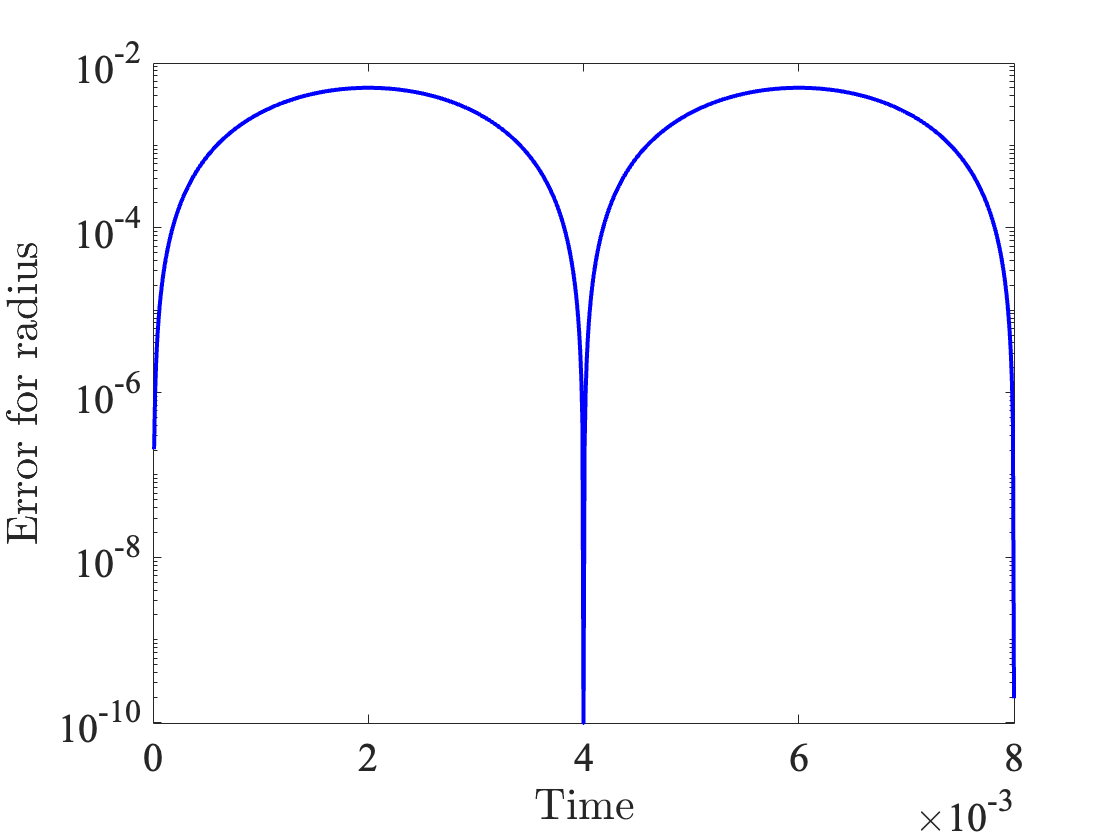}\label{fig:radius-error}}
\caption{The evolution dynamics for the circular case. In (a), the profiles of $\mathbf x(t)$ at $t=0$, $0.001$, $0.003$ are shown; (b) The errors for radius $R(t)$ are plotted at a time interval $[0,8\times 10^{-3}]$.}
\label{fig:circle}
\end{figure}

{\color{black} Finally, we examine the time step constraint for numerical stability. With the same initial condition \eqref{eq:IC1} given by the unit circle, to ensure that the boundary spacing $\frac{2\pi R(t)}{N}$ remains constant in time (i.e., the boundary remains a unit circle), we consider the free boundary equation \eqref{eq:FBP_circle} without the prescribed time-dependent forcing term, that is, $h(t)=0$, which is equivalent to the free boundary equation \eqref{BIM}. For different number of boundary points, $N=100$, $200$, $400$, $800$, we identify the largest time step $\Delta t_{\max}^{\mathrm{stable}}$ using the Simpson's rule for numerical integration and Forward Euler method for time discretization, for which the numerical solution remains stable (i.e., $\|e_V\|_{L^2}\leq 10^{-2}$) up to a fixed final time $t=0.1$. The results in Table~\ref{tab:timestep_constraint} show that finer boundary discretizations require smaller time steps for stable computations. We note that the largest stable time step may vary for different geometries.
\begin{table}[htbp]
\centering
{\color{black}
\caption{Empirical time step constraint for ensuring stability of the forward Euler method.}
\label{tab:timestep_constraint}
\begin{tabular}{|c|c|}
\hline
Number of boundary points $N$  &
Largest stable time step $\Delta t_{\max}^{\mathrm{stable}}$ \\
\hline
100 & $3.0\times10^{-4}$ \\
\hline
200 & $1.0\times10^{-4}$ \\
\hline
400 & $3.0\times10^{-5}$ \\
\hline
800 & $9.5\times10^{-6}$ \\
\hline
\end{tabular}}
\end{table}
}

\subsection{Free boundary equation on perturbations of a circle}\label{circle}
In this section, we consider the free boundary equation \eqref{BIM} with the initial condition given by perturbations of a circle,
\begin{equation}\label{eqn:perturbed-circle}
\begin{aligned}
    r(\theta) &=r_0+D_1\cos(D_2\theta),\\
\mathbf x(0)&=\Big(x_1(0),x_2(0)\Big)=\Big(r(\theta)\cos(\theta),r(\theta)\sin(\theta)\Big),\quad \theta\in[0,2\pi],
\end{aligned}
\end{equation}
where $D_1$ and $D_2$ are parameters. First, we can observe that the area of $\Omega(t)$ is conserved for \eqref{hele-shaw} with $f=0$,

\[
\frac{d}{dt} |\Omega(t)| = \int_{\Gamma} V_n(\mathbf x) dS_{\mathbf x}=- \int_{\Gamma} \frac{\partial p(\mathbf x)}{\partial \mathbf n(\mathbf x)} \, dS_{\mathbf x}=-\int_{\Omega} \Delta p \, d\Omega=0.
\]
Thus the area of the steady-state circle with the initial condition given in \eqref{eqn:perturbed-circle} is
\[\text{Area} = \frac{1}{2} \int_0^{2\pi} \Big( r_0 + D_1\cos(D_2\theta) \Big)^2 \, d\theta
= \pi r_0^2 + \frac{\pi}{2} D_1^2,\]
which implies the radius of the steady-state circle is
$r_s=(r_0^2+D_1^2/2)^{1/2}$.

In Fig~\ref{fig:perturbed-circle}, we show the numerical result for the case of $D_1=0.3$, and $D_2=5$. See Appendix \ref{SM4} for results with other values of $D_1$ and $D_2$. In this experiment, we apply the trapezoidal rule to approximate $A$, Simpson's rule to approximate $b$, and the forward Euler method with $\Delta t=10^{-5}$ and $N=400$, and $r_0=1$. We point out that in this case, we redistribute the mesh points onto an equispaced mesh using cubic spline interpolation to prevent point collapse during boundary evolution. The results are summarized in Figure \ref{fig:perturbed-circle}. Panel (a) shows the boundary evolution dynamics for different perturbations of the circle in \eqref{eqn:perturbed-circle}, with initial conditions shown as black curves. The results indicate that the boundaries approach a circle of radius $r_s$, as predicted by the conservation of area. Furthermore, we examine the maximum and minimum distances from the boundary points to the center, defined as $d_i = \|\mathbf x_i-\mathbf x_*\|$, $i=1,\cdots,N$ with $\mathbf x_*=\frac{\int_\Gamma \mathbf x dS_{\mathbf x}}{\int_\Gamma dS_{\mathbf x}}$ being the center of $\Omega(t)$, and compare them with $r_s$ in Figure \ref{fig:perturbed-circle} (b). This quantitatively demonstrates that the radii of the estimated curves converge to $r_s$. 

\begin{figure}[htbp]
\center
\subfigure[Evolution dynamics for $D_1=0.3$, $D_2=5$.]{\includegraphics[width=0.45\textwidth]{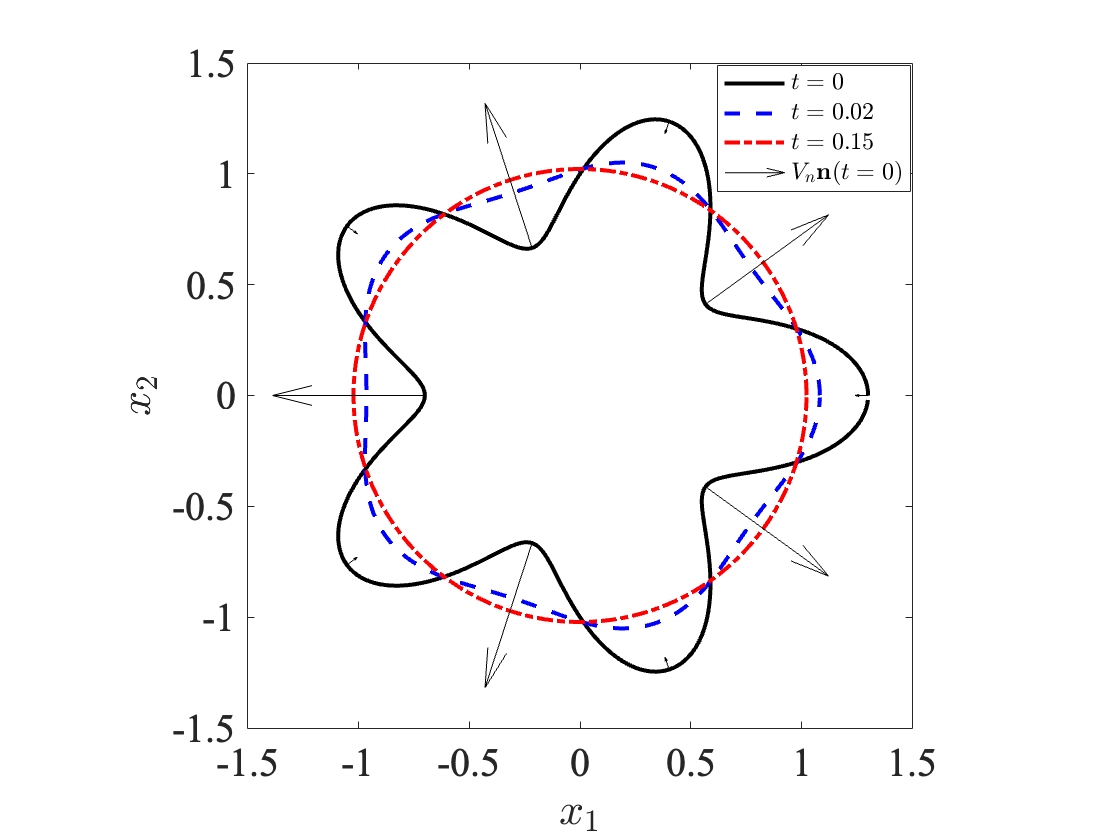}}
\subfigure[Radius evolution for $D_1=0.3$, $D_2=5$ compared with steady-state value.]{\includegraphics[width=0.45\textwidth]{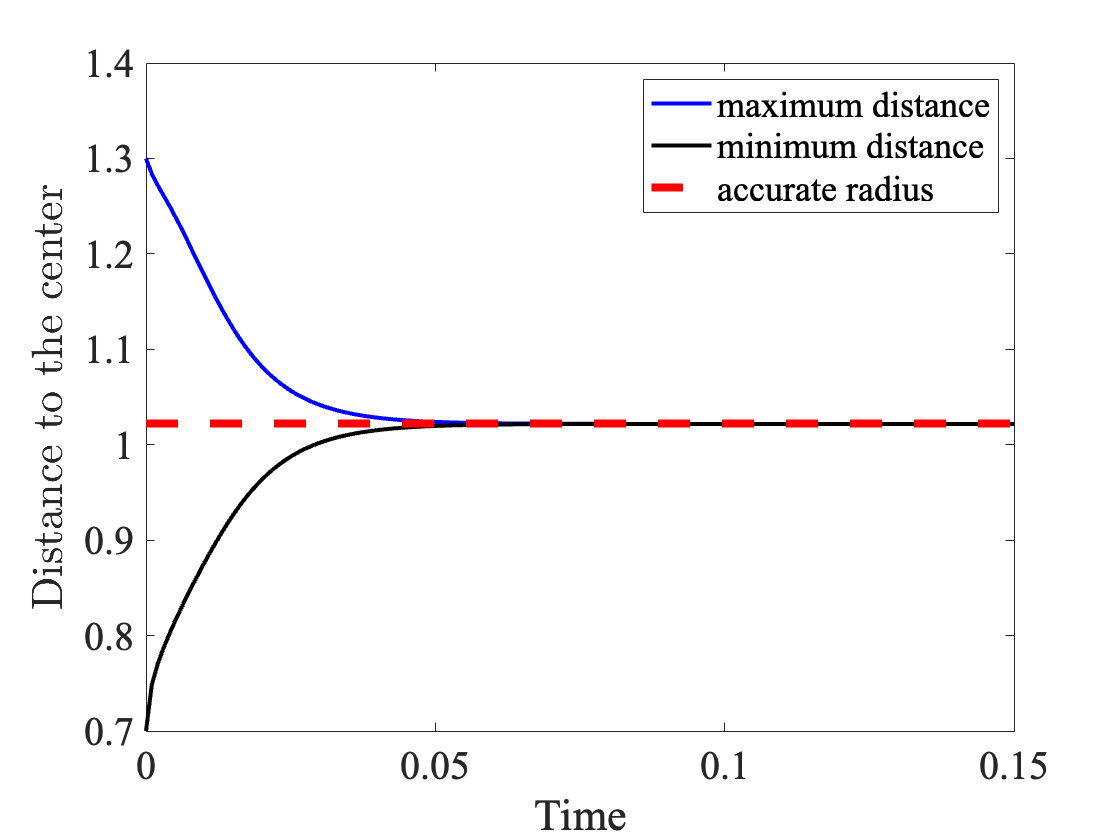}}
\caption{The evolution dynamics for perturbed circles with different $D_1$ and $D_2$. In (a),  the profiles of $\mathbf x(t)$ at different times and boundary motion $ V_n \mathbf n$ at $t=0$ are shown. In (b), comparisons of the maximum and minimum distances from the boundary points to the center with the steady-state circle radius $r_s$ are shown, where we observe that the perturbed circle evolves into the circle with radius $r_s$. }
\label{fig:perturbed-circle}
\end{figure}

\subsection{Free boundary equation on smooth closed curves}
In this section, we apply our method to smooth closed curves, including a heart-shaped curve and a roughly humanoid-shaped curve. Specifically, we use the trapezoidal rule to approximate $A$, Simpson's rule to approximate $b$, and the forward Euler method with $\Delta t=10^{-5}$, $N=400$, and redistribute mesh points onto an equispaced mesh using cubic spline interpolation to prevent point collapse during boundary evolution.

First, we consider the heart-shaped curve, whose initial condition is given as
\begin{equation*}\label{eq:heart}
\mathbf x(0)=\Big(x_1(0),x_2(0)\Big)=\Big(\sin(\theta),1.5\cos(\theta)-0.4\cos(2\theta)-0.1\cos(3\theta)-0.1\cos(4\theta)\Big),\quad \theta\in[0,2\pi].
\end{equation*}
Furthermore, we study the evolution dynamics of a roughly humanoid-shaped curve. The curve is initially defined by 25 points, which are interpolated to generate 400 points that serve as its initial configuration. The results are summarized in Figure \ref{fig:smooth-closed-curve}. Figure \ref{fig:smooth-closed-curve} (a) and (c) present the boundary evolution dynamics for different initial conditions shown as black curves. The results indicate that the boundaries approach a circle, which is quantitatively demonstrated in Figure \ref{fig:smooth-closed-curve} (b) and (d) by comparing the maximum and minimum distances from the boundary points to the center.

\begin{figure}[htbp]
\center
\subfigure[Evolution dynamics for heart-shaped curve.]{\includegraphics[width=0.45\textwidth]{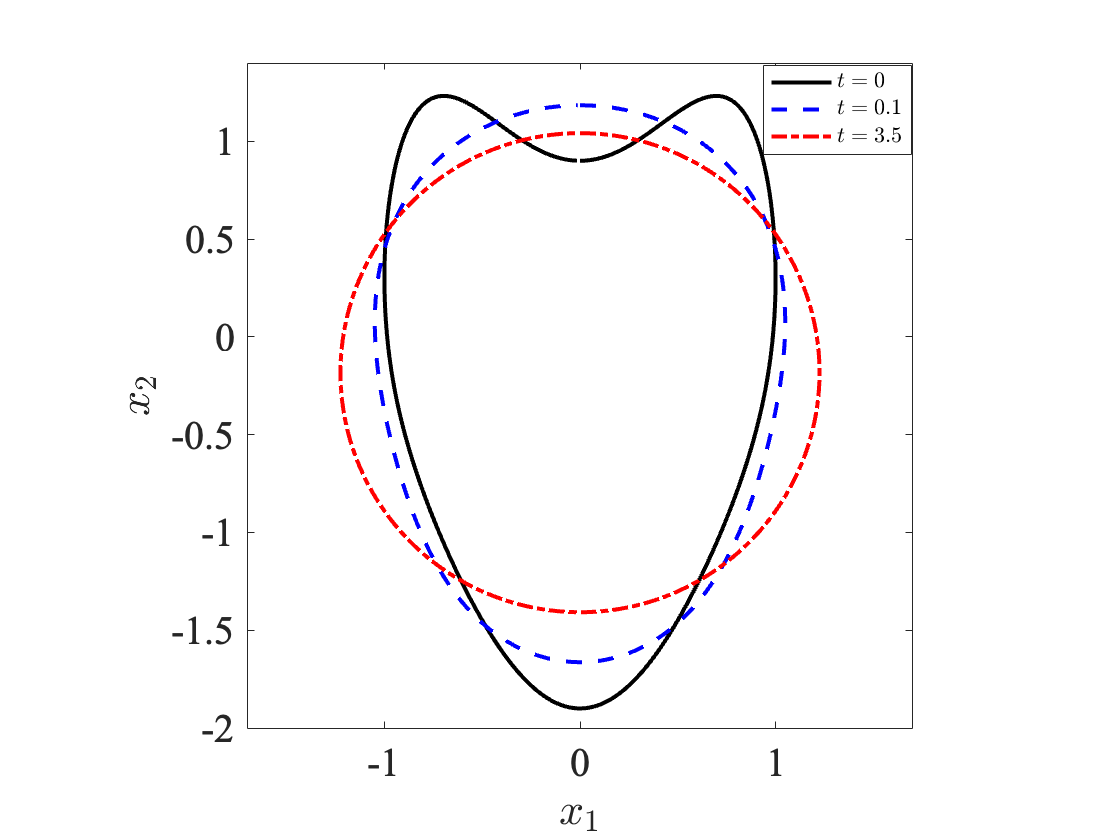}}
\subfigure[Radius evolution for heart-shaped curve.]{\includegraphics[width=0.45\textwidth]{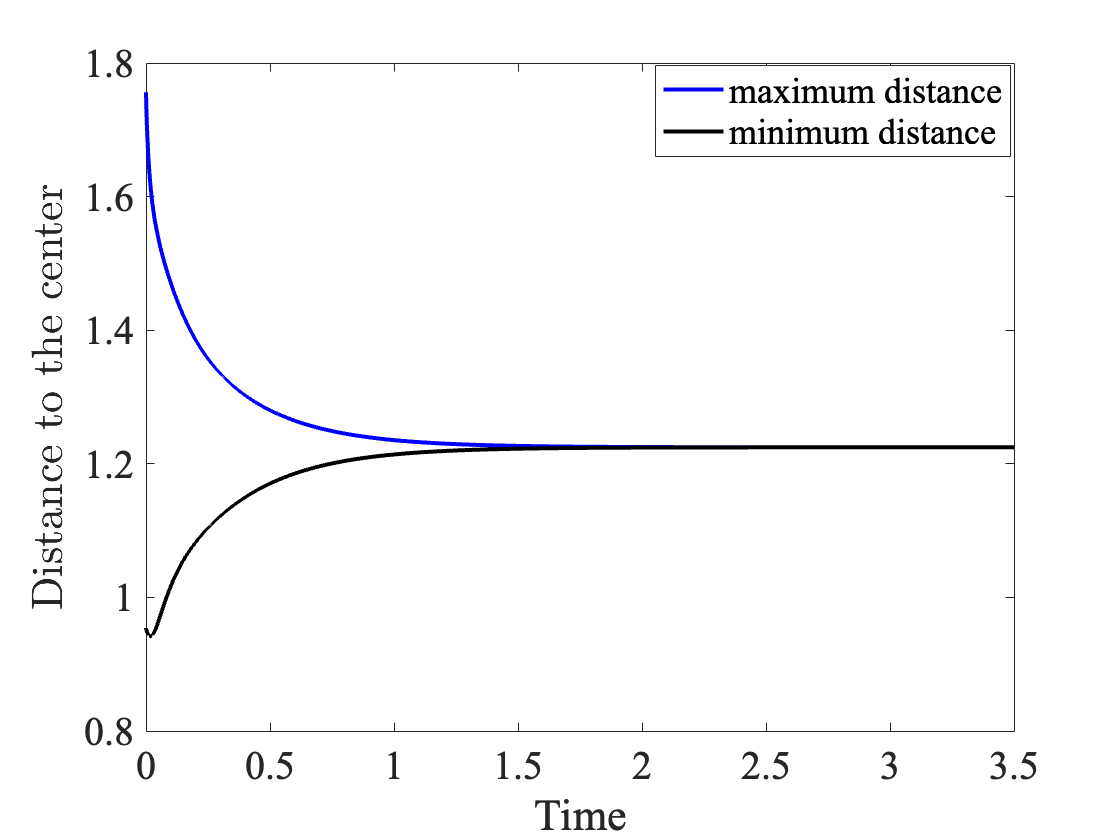}}
\subfigure[Evolution dynamics for roughly humanoid-shaped curve.]{\includegraphics[width=0.45\textwidth]{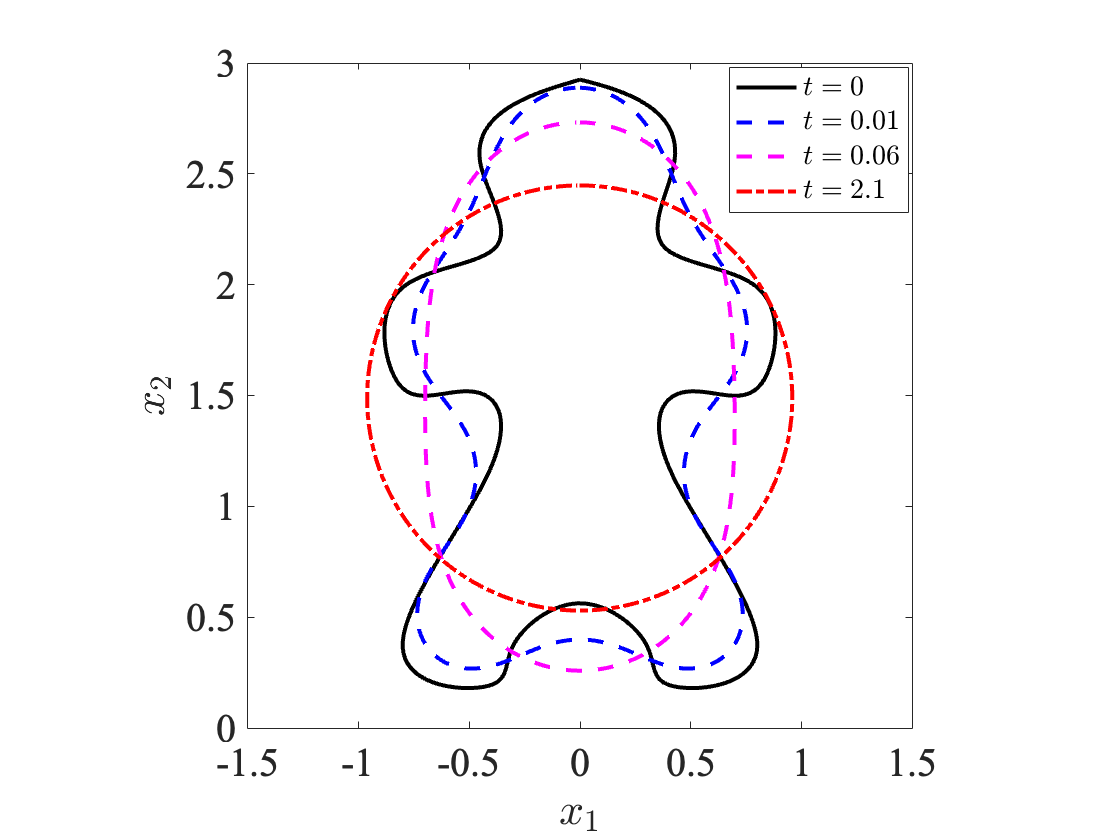}}
\subfigure[Radius evolution for roughly humanoid-shaped curve.]{\includegraphics[width=0.45\textwidth]{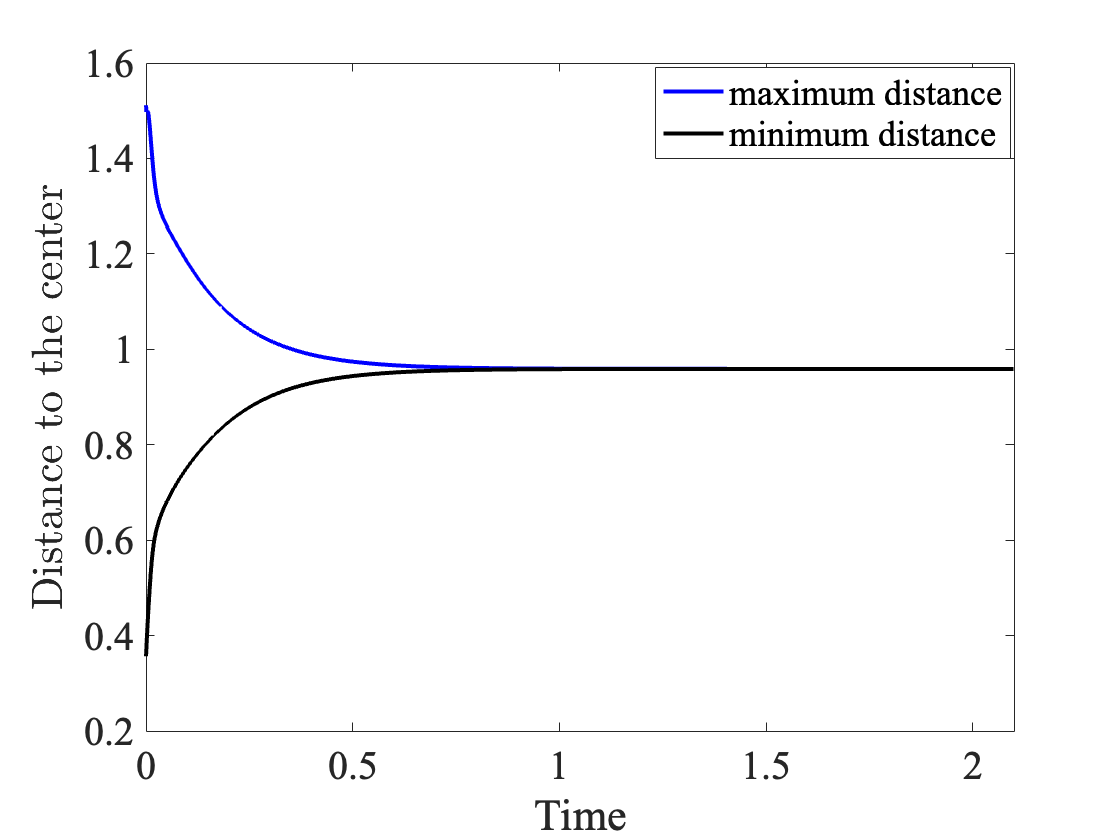}}
\caption{The evolution dynamics for smooth closed curves. In (a) and (c), the profiles of $\mathbf x(t)$ at different times are shown. In (b) and (d), the maximum and minimum distances from the boundary points to the center are shown, where we observe that the smooth closed curves evolve into circles. }
\label{fig:smooth-closed-curve}
\end{figure}

\section{Conclusion}\label{sec6}

In this paper, we have introduced a novel meshfree computational framework for solving the two-dimensional Hele-Shaw problem with surface tension. The core of our approach is a geometric local parameterization of the moving boundary via GMLS, enabling accurate and robust treatment of complex, evolving geometries represented solely by point clouds.

Our main contributions are as follows
1) \textbf{GMLS-based Discretization:} We developed a numerical scheme that systematically discretizes the singular boundary integral formulation of the Hele-Shaw problem. By constructing local charts via GMLS, we accurately approximate geometric quantities such as normal vectors and curvature, and handle the singular integrals arising from the Green's function 
by analytically isolating the logarithmic singularity and applying local quadratic interpolation to the rest to maintain second-order accuracy.
    2) \textbf{Theoretical Analysis:} We provided a rigorous theoretical foundation for the proposed method. This includes a detailed error analysis demonstrating the spatial consistency of our discretization for the boundary integral operators and establishing the invertibility of the resulting linear system under appropriate conditions. Our analysis reveals that the convergence rate is determined by the smoothness of the underlying boundary and the order of the quadrature rules that are used in approximating the boundary integral equation.
    3) \textbf{Numerical Verification:} We validated the effectiveness and accuracy of our method through a series of numerical experiments. Convergence studies on a circle confirmed the high-order spatial accuracy and the expected temporal convergence of the time discretization schemes. Furthermore, simulations on perturbed circles and more complex smooth curves (e.g., heart-shaped and humanoid-shaped boundaries) demonstrated the method's robustness and its ability to capture the correct long-time dynamics, where the evolving boundaries, besides circles, as predicted by area conservation and surface tension effects.

The proposed method offers a powerful and flexible alternative to traditional parameterized boundary integral methods, particularly for problems where maintaining a global parameterization is difficult or where complex topological changes may occur in future extensions. Several directions for future research follow from this work. A natural extension is to handle problems with a source term, namely, \( f(\mathbf{x}) \neq 0 \), which arise in biological applications like tumor growth. Developing adaptive mesh refinement strategies would also enhance efficiency for long-time simulations or boundaries with high-curvature regions. Finally, the most significant challenge and opportunity is the generalization to three-dimensional Hele-Shaw problems and other free-boundary problems, where the advantages of avoiding global parameterization would be particularly impactful.

\section*{Acknowledgment}

The research of JH was partially supported under the NSF grants DMS-2209535 and DMS-2505605, and the ONR grant N00014-22-1-2193.  WH was supported by the National Institute of General Medical Sciences through grant 1R35GM146894. ZZ was partially supported under the ICDS seed grant.

\appendix

\section{Component of A}\label{SM1}

In this appendix, we report components of the matrix $\mathbf{A}$ based on the discretization proposed in Section~\ref{quadrature_rule}.

Specifically, components of $\mathbf{A}$ are:

\begin{equation*}\footnotesize
    \begin{aligned}
        \mathbf A_{i,i}=&-\frac{1}{2\pi}\Big(-\Delta s_{-i}\ln(-\Delta s_{-i})+\Delta s_{-i}+\Delta s_i\ln(\Delta s_i)-\Delta s_i\Big)\\
        &-\frac{1}{2\pi}\left(\frac{\Delta s_i + \Delta s_{-i}}{\Delta s_i \Delta s_{-i}}\right)\Big(-\frac{1}{2}(\Delta s_{-i})^2\ln(-\Delta s_{-i})+\frac{1}{4}(\Delta s_{-i})^2+\frac{1}{2}(\Delta s_i)^2\ln(\Delta s_i)-\frac{1}{4}(\Delta s_i)^2\Big)\\
        &-\frac{1}{2\pi}\left(\frac{1}{\Delta s_i \Delta s_{-i}}\right)\Big(-\frac{1}{3}(\Delta s_{-i})^3\ln(-\Delta s_{-i})+\frac{1}{9}(\Delta s_{-i})^3+\frac{1}{3}(\Delta s_i)^3\ln(\Delta s_i)-\frac{1}{9}(\Delta s_i)^3\Big),\\
        \mathbf A_{i,i-1}=&\frac{1}{2\pi}\left(\frac{\Delta s_i \sqrt{1 + \big(p_i'(\Delta s_{-i})\big)^2}}{\Delta s_{-i} (\Delta s_{-i} - \Delta s_i)} \right)\Big(-\frac{1}{2}(\Delta s_{-i})^2\ln(-\Delta s_{-i})+\frac{1}{4}(\Delta s_{-i})^2+\frac{1}{2}(\Delta s_i)^2\ln(\Delta s_i)-\frac{1}{4}(\Delta s_i)^2\Big)\\
        &\frac{1}{2\pi}\left(\frac{\sqrt{1 + \big(p_i'(\Delta s_{-i})\big)^2}}{\Delta s_{-i} (\Delta s_i - \Delta s_{-i})}\right)\Big(-\frac{1}{3}(\Delta s_{-i})^3\ln(-\Delta s_{-i})+\frac{1}{9}(\Delta s_{-i})^3+\frac{1}{3}(\Delta s_i)^3\ln(\Delta s_i)-\frac{1}{9}(\Delta s_i)^3\Big)\\
        &\frac{1}{4\pi}\ln\left(1+\frac{p_i^2(\Delta s_{-i})}{(\Delta s_{-i})^2}\right)\sqrt{1+\big(p_i'(\Delta s_{-i})\big)^2}\frac{\Delta s_{-i}}{2}-G(\mathbf x_i,\mathbf x_{i-1})\sqrt{1+\big(p_{i-2}'(\Delta s_{i-2})\big)^2}\frac{\Delta s_{i-2}}{2},\\
        \mathbf A_{i,i+1}=&-\frac{1}{2\pi}\left(\frac{\Delta s_{-i} \sqrt{1 + \big(p_i'(\Delta s_{i})\big)^2}}{\Delta s_{i} (\Delta s_{-i} - \Delta s_i)} \right)\Big(-\frac{1}{2}(\Delta s_{-i})^2\ln(-\Delta s_{-i})+\frac{1}{4}(\Delta s_{-i})^2+
        \frac{1}{2}(\Delta s_i)^2\ln(\Delta s_i)-\frac{1}{4}(\Delta s_i)^2\Big)\\
        &-\frac{1}{2\pi}\left(\frac{\sqrt{1 + \big(p_i'(\Delta s_{i})\big)^2}}{\Delta s_{i} (\Delta s_i - \Delta s_{-i})}\right)\Big(-\frac{1}{3}(\Delta s_{-i})^3\ln(-\Delta s_{-i})+\frac{1}{9}(\Delta s_{-i})^3+\frac{1}{3}(\Delta s_i)^3\ln(\Delta s_i)-\frac{1}{9}(\Delta s_i)^3\Big)\\
        &-\frac{1}{4\pi}\ln\left(1+\frac{p_i^2(\Delta s_{i})}{(\Delta s_{i})^2}\right)\sqrt{1+\big(p_i'(\Delta s_{i})\big)^2}\frac{\Delta s_{i}}{2}-G(\mathbf x_i,\mathbf x_{i+1})\frac{\Delta s_{i+1}}{2},\\
        \mathbf A_{i,j}=&G(\mathbf x_i,\mathbf x_j)\frac{\Delta s_j}{2} -+G(\mathbf x_i,\mathbf x_j)\sqrt{1+\big(p'_{j-1}(\Delta s_{j-1})\big)^2}\frac{\Delta s_{j-1}}{2},\quad j\neq i-1,i,i+1.
    \end{aligned}
\end{equation*}

\section{Proof of Proposition \ref{prop4.2}}\label{SM2}
In this appendix, we report the additional proof in Proposition~\ref{prop4.2} that bound 
$|I_N^* - \tilde{I}_N|$.

\begin{eqnarray}
|I_N^* - \tilde{I}_N| &\leq& |I_N^* - I_N | + |I_N - \tilde{I}_N| \notag \\ &=&
\frac{1}{4\pi}\left|\int_{\Delta s(\mathbf x_{-i})}^{\Delta s(\mathbf x_i)} \ln\left(1+\frac{p^2(s)}{s^2}\right)\psi^*(s)ds - \int_{\Delta s_{-i}}^{\Delta s_i} \ln\left(1+\frac{p_i^2(s)}{s^2}\right)\psi_i(s)ds\right| \notag\\
&&+\sum_{j\neq i,i-1}\left|\int_0^{\Delta s(\mathbf{x}_j)} G\big(\mathbf x_i,\iota(s)\big)\psi^*(s)ds- \int_0^{\Delta s_j} G\big(\mathbf x_i,\iota_j(s)\big)\psi_j(s)ds\right|+\mathcal{O}\left(\left(\frac{\log N}{N}\right)^{q}\right),\notag
\end{eqnarray}
where $\tilde{I}_N$ is the quadrature rule approximation to $I_N$ with error of order $q$.

By Proposition~\ref{prop3.1} and \eqref{Errorpsi}, we note the following bounds for the integrands. With probability higher than $1-\frac{3}{N}$,
\BEA
\left| \varphi^*(s) - \varphi_i(s)\right|&:=&\left|\ln\left(1+\frac{p^2(s)}{s^2}\right)\psi^*(s) - \ln\left(1+\frac{p_i^2(s)}{s^2}\right)\psi_i(s)\right| \notag \\ 
&\leq& \left|\ln\left(1+\frac{p^2(s)}{s^2}\right)\psi^*(s)-\ln\left(1+\frac{p_i^2(s)}{s^2}\right)\psi^*(s)\right|
+\left| \ln\left(1+\frac{p_i^2(s)}{s^2}\right)\psi^*(s)-\ln\left(1+\frac{p_i^2(s)}{s^2}\right)\psi_i(s)\right|\notag\\
&=& \left|\frac{1}{1+ p^2(s)/s^2}\right| \left|\frac{p^2(s) -p_i^2(s)}{s^2} \right|\left|\psi^*(s)\right|+\mathcal{O}\left(\left(\frac{\log N}{N}\right)^{\ell}\right) 
= \mathcal{O}\left(\frac{\big(\log N\big)^{\ell+1}}{N^{\ell-3}}\right), \notag
\EEA
where we have used a lower bound for the separation distance, $|s|\geq \min_{i\neq j} d_g(\mathbf{x}_i,\mathbf{x}_j) = q_{X,\Gamma} \geq N^{-2}$ with probability higher than $1-\frac{1}{N}$ (see
Lemma A.2 in \cite{yan2023spectral} for the detailed proof).

Similarly, for $j\neq i,i-1$, with probability higher than $1-\frac{3}{N}$,
\BEA
\left|G\big(\mathbf x_i,\iota(s)\big)\psi^*(s) - G\big(\mathbf x_i,\iota_j(s)\big)\psi_j(s)\right|&\leq&\left|G\big(\mathbf x_i,\iota(s)\big)-G\big(\mathbf x_i,\iota_j(s)\big)\right|\left|\psi^*(s)\right|  + \left|G\big(\mathbf x_i,\iota_j(s)\big)\right|\left|\psi^*(s)-\psi_j(s)\right| \notag\\
&=& \frac{1}{2\pi}\frac{\left|\big(\mathbf{x}_i -\iota(s)\big)\cdot\big(\iota(s) - \iota_j(s)\big)\right|}{\|\mathbf{x}_i -\iota(s)\|^2}\left|\psi^*(s)\right| + \mathcal{O}\left(\left(\frac{\log N}{N}\right)^{\ell}\right) \notag\\
&=& \mathcal{O}\left(\frac{\big(\log N\big)^{\ell+1}}{N^{\ell-1}}\right).\notag
\EEA

Then, one can deduce that with probability higher than $1-\frac{8}{N}$,
\BEA
\left|\int_{\Delta s(\mathbf x_{-i})}^{\Delta s(\mathbf x_i)} \varphi^*(s)ds -\int_{\Delta s_{-i}}^{\Delta s_i} \varphi_i(s)ds \right| &\leq& \left| \int_{\Delta s(\mathbf x_{-i})}^{\Delta s(\mathbf x_i)} \varphi^*(s)ds-\int_{\Delta s_{-i}}^{\Delta s_i}\varphi^*(s)ds\right|  + \left| \int_{\Delta s_{-i}}^{\Delta s_i} \big(\varphi^*(s)ds - \varphi_i(s)\big)ds\right| \notag \\
&\leq&C \big(|\Delta s_i - \Delta s(\mathbf{x}_i)| + |\Delta s_{-i} - \Delta s(\mathbf{x}_{-i})|\big) + \max_{s\in [\Delta s_{-i},\Delta s_{i}]}\left|\varphi^*(s)-\varphi_i(s)\right| \left|\Delta s_i-\Delta s_{-i}\right|\notag\\
&=&\mathcal{O}\left(\left(\frac{\log N}{N}\right)^{\ell+1}\right) + \mathcal{O}\left(\frac{\big(\log N\big)^{\ell+2}}{N^{\ell-2}}\right),\notag
\EEA
and with probability higher than $1-\frac{6}{N}$,
\BEA
\left|\int_0^{\Delta s(\mathbf{x}_j)} G\big(\mathbf x_i,\iota(s)\big)\psi^*(s)ds- \int_0^{\Delta s_j} G\big(\mathbf x_i,\iota_j(s)\big)\psi_j(s)ds\right| &=& \mathcal{O}\left(\left(\frac{\log N}{N}\right)^{\ell+1}\right) \notag + \mathcal{O}\left(\frac{\big(\log N\big)^{\ell+2}}{N^{\ell}}\right).\notag
\EEA
Therefore, with probability higher than $1-\frac{14}{N}$,
\BEA
|I_N^*-I_N| &\leq& \left|\int_{\Delta s(\mathbf x_{-i})}^{\Delta s(\mathbf x_i)} \varphi^*(s)ds -\int_{\Delta s_{-i}}^{\Delta s_i} \varphi_i(s)ds \right| \notag+ \sum_{j\neq i,i-1} \left|\int_0^{\Delta s(\mathbf{x}_j)} G\big(\mathbf x_i,\iota(s)\big)\psi^*(s)ds- \int_0^{\Delta s_j} G\big(\mathbf x_i,\iota_j(s)\big)\psi_j(s)ds\right|\notag \\
&=& \mathcal{O}\left(\frac{\big(\log N\big)^{\ell+2}}{N^{\ell-2}}\right) + \big(N-2\big) \mathcal{O}\left(\frac{\big(\log N\big)^{\ell+2}}{N^{\ell}}\right) 
= \mathcal{O}\left(\frac{\big(\log N\big)^{\ell+2}}{N^{\ell-2}}\right). \notag
\EEA
Finally, we have
\BEA
|I_N^*-\tilde{I}_N| \leq |I_N^*-I_N|+ |I_N-\tilde{I}_N| = \mathcal{O}\left(\frac{\big(\log N\big)^{\ell+2}}{N^{\ell-2}}\right) + \mathcal{O}\left(\left(\frac{\log N}{N}\right)^{q}\right)\label{ErrorIn}.
\EEA
With the bounds in \eqref{ErrorIs} and \eqref{ErrorIn}, the proof is completed.

\section{Proof of Proposition~\ref{prop4.3}}\label{SM3}

For simplicity, we define
\[
b(\mathbf{x}_i)= -\frac{\kappa(\mathbf x_i)}{2} + \frac{1}{2}I_{b}^*, 
\]
where
\[I^*_b=-\sum_{j}\int_{\Delta s(\mathbf x_{-j})}^{\Delta s(\mathbf x_j)} \kappa\big(\iota(s)\big)\Big(\nabla G\big(\mathbf x_i,\iota(s)\big)\cdot \mathbf{n}\big(\iota(s)
\big)\Big)\sqrt{|\iota'(s)|}ds.
\]
     By Proposition~\ref{prop3.1} and Proposition~\ref{prop4.1}, with probability higher than $1-\frac{1}{N}$,
     \begin{equation}\label{Errorkappa}
     \left|\kappa(\mathbf x_i)-\kappa_i\right|=\mathcal{O}\left(\left(\frac{\log N}{N}\right)^{\ell-1}\right).
     \end{equation}
Then we consider the error bond for $\left|I^*_b-\tilde{I}_b\right|$. Note that
\begin{eqnarray}
        \left|I^*_b-\tilde{I}_b\right|&\leq& \left|I^*_b-I_b\right|+\left|I_b-\tilde{I}_b\right|\notag\\
    &=&\sum_{j}\left|\int_{\Delta s(\mathbf x_{-j})}^{\Delta s(\mathbf x_j)} \Big(\nabla G\big(\mathbf x_i,\iota(s)\big)\cdot \mathbf{n}\big(\iota(s)
\big)\Big)\phi^*(s)ds-\int_{\Delta s_{-j}}^{\Delta s_j} \Big(\nabla G\big(\mathbf x_i,\iota_j(s)\big)\cdot \hat{\mathbf{n}}\big(\iota_j(s)
\big)\Big)\phi_j(s)ds\right|+\mathcal{O}\left(\left(\frac{\log N}{N}\right)^{q}\right),\notag
\end{eqnarray}
where $\tilde{I}_b$ is the quadrature rule approximation to $I_b$ with error of order-$q$, and $\phi^*(s):=\kappa\big(\iota(s)\big)\sqrt{|\iota'(s)|}$, $\phi_j(s):=\kappa\big(\iota_j(s)\big)\sqrt{1+\big(p_j'(s)\big)^2}$.

By Proposition~\ref{prop3.1}, with probability higher than $1-\frac{2}{N}$.
\begin{eqnarray}
    \left|\phi^*(s)-\phi_j(s)\right|&=&\left|\kappa\big(\iota(s)\big)\sqrt{|\iota'(s)|}-\kappa\big(\iota_j(s)\big)\sqrt{1+\big(p_j'(s)\big)^2}\right|\notag\\
    &\leq& \left|\kappa\big(\iota(s)\big)\right|\left|\sqrt{|\iota'(s)|}-\sqrt{1+\big(p'_i(s)\big)^2}\right|+\sqrt{1+\big(p'_i(s)\big)^2}\left|\kappa\big(\iota(s)\big)-\kappa_i(s)\right|\notag\\
&=&\mathcal{O}\left(\left(\frac{\log N}{N}\right)^{\ell}\right) + \mathcal{O}\left(\left(\frac{\log N}{N}\right)^{\ell-1}\right)=\mathcal{O}\left(\left(\frac{\log N}{N}\right)^{\ell-1}\right).\notag
\end{eqnarray}

Then, with probability higher than $1-\frac{4}{N}$,
\begin{eqnarray}
   &&\left|\nabla G\big(\mathbf x_i,\iota(s)\big)\cdot\mathbf{n}\big(\iota(s)\big)\phi^*(s)-\nabla G\big(\mathbf x_i,\iota_j(s)\big)\cdot\hat{\mathbf n}\big(\iota_j(s)\big)\phi_j(s)\right|\notag\\
    &&\leq\left|\nabla G\big(\mathbf x_i,\iota(s)\big)\cdot\mathbf{n}\big(\iota(s)\big)-\nabla G\big(\mathbf x_i,\iota_j(s)\big)\cdot\hat{\mathbf n}\big(\iota_j(s)\big)\right|\left|\phi^*(s)\right|+\left|\nabla G\big(\mathbf x_i,\iota_j(s)\big)\cdot\hat{\mathbf n}\big(\iota_j(s)\big)\right|\left|\phi^*(s)-\phi_j(s)\right|\notag\\
    &&\leq\left|\nabla G\big(\mathbf x_i,\iota(s)\big)\cdot\left(\mathbf {n}\big(\iota(s)\big)-\hat{\mathbf n}\big(\iota_j(s)\big)\right)\right| \left|\phi^*(s)\right| +\left|\left(\nabla G\big(\mathbf x_i,\iota(s)\big)-\nabla G\big(\mathbf x_i,\iota_j(s)\big)\right)\cdot\hat{\mathbf n}\big(\iota_j(s)\big)\right|\left|\phi^*(s)\right|+\mathcal{O}\left(\left(\frac{\log N}{N}\right)^{\ell-1}\right)\notag\\
    &&=C\frac{\left|\big(\mathbf{x}_i -\iota(s)\big)\cdot\big(\mathbf n\big(\iota(s)\big) - \hat{\mathbf n}\big(\iota_j(s)\big)\big)\right|}{\|\mathbf{x}_i -\iota(s)\|^2}+C\left|\frac{-\|\mathbf x_i-\iota(s)\|^2\mathbf {1}_2+2\big(\mathbf x_i-\iota(s)\big)\odot\big(\mathbf x_i-\iota(s)\big)}{\|\mathbf x_i-\iota(s)\|^4}\odot\big(\iota(s)-\iota_j(s)\big)\cdot\hat{\mathbf n}\big(\iota_j(s)\big)\right|\notag\\
    &&\quad+\mathcal{O}\left(\left(\frac{\log N}{N}\right)^{\ell-1}\right)\notag\\
&&=\mathcal{O}\left(\frac{\big(\log N\big)^{\ell}}{N^{\ell-2}}\right)+\mathcal{O}\left(\frac{\big(\log N\big)^{\ell+1}}{N^{\ell-3}}\right)+\mathcal{O}\left(\left(\frac{\log N}{N}\right)^{\ell-1}\right)=\mathcal{O}\left(\frac{\big(\log N\big)^{\ell+1}}{N^{\ell-3}}\right)\notag
\end{eqnarray}
where $\mathbf{1}_2 = \begin{bmatrix} 1,1\end{bmatrix}^\top$ and $\odot$ represents the Hadamard product.

Therefore, with probability higher than $1-\frac{9}{N}$,
\begin{eqnarray}
    \left| I^*_b-I_b\right|&\leq&\sum_{j}\left|\int_{\Delta s(\mathbf x_{-j})}^{\Delta s(\mathbf x_j)} \Big(\nabla G\big(\mathbf x_i,\iota(s)\big)\cdot \mathbf{n}\big(\iota(s)
\big)\Big)\phi^*(s)ds-\int_{\Delta s_{-j}}^{\Delta s_j} \Big(\nabla G\big(\mathbf x_i,\iota(s)\big)\cdot \mathbf{n}\big(\iota(s)
\big)\Big)\phi^*(s)ds\right|\notag\\
&&+\sum_j\left|\int_{\Delta s_{-j}}^{\Delta s_j} \Big(\nabla G\big(\mathbf x_i,\iota(s)\big)\cdot \mathbf{n}\big(\iota(s)
\big)\Big)\phi^*(s)ds-\int_{\Delta s_{-j}}^{\Delta s_j} \Big(\nabla G\big(\mathbf x_i,\iota_j(s)\big)\cdot \hat{\mathbf{n}}\big(\iota_j(s)
\big)\Big)\phi_j(s)ds\right|\notag\\
&=&C \big(|\Delta s_i - \Delta s(\mathbf{x}_i)|+ |\Delta s_{-i} - \Delta s(\mathbf{x}_{-i})|\big)\notag\\
&&+ \max_{s\in [\Delta s_{-i},\Delta s_{i}]}\left|\nabla G\big(\mathbf x_i,\iota(s)\big)\cdot\mathbf{n}\big(\iota(s)\big)\phi^*(s)-\nabla G\big(\mathbf x_i,\iota_j(s)\big)\cdot\hat{\mathbf n}\big(\iota_j(s)\big)\phi_j(s)\right| \left|\Delta s_i-\Delta s_{-i}\right|\notag\\
&=&\mathcal{O}\left(\left(\frac{\log N}{N}\right)^{\ell+1}\right)+\mathcal{O}\left(\frac{\big(\log N\big)^{\ell+2}}{N^{\ell-2}}\right)=\mathcal{O}\left(\frac{\big(\log N\big)^{\ell+2}}{N^{\ell-2}}\right).\notag
\end{eqnarray}

Finally, we have
\BEA
|I_b^*-\tilde{I}_b| \leq |I_b^*-I_b|+ |I_b-\tilde{I}_b| = \mathcal{O}\left(\frac{\big(\log N\big)^{\ell+2}}{N^{\ell-2}}\right) + \mathcal{O}\left(\left(\frac{\log N}{N}\right)^{q}\right)\label{ErrorIb}.
\EEA
With the bounds in \eqref{Errorkappa} and \eqref{ErrorIb}, the proof is completed.

\section[Picture]{Additional Numerical experiments for Section \ref{circle}}\label{SM4}

We present additional simulations of the free boundary equation \eqref{BIM} using the perturbed circular initial condition \eqref{eqn:perturbed-circle} for different values of $(D_1, D_2)$. All numerical settings follow those in Section~\ref{circle}. Figure~\ref{fig:perturbed-circle2} illustrates the boundary evolution for several representative parameter choices. 
Unlike the results shown in Figure~\ref{fig:perturbed-circle}, these results are obtained without employing the cubic spline interpolation that redistributes the mesh points onto an equispaced mesh since these points do not collapse during the boundary evolution.

\begin{figure}[htb!]
\center
\subfigure[Evolution dynamics for $D_1=0.1$, $D_2=3$.]{\includegraphics[width=0.45\textwidth]{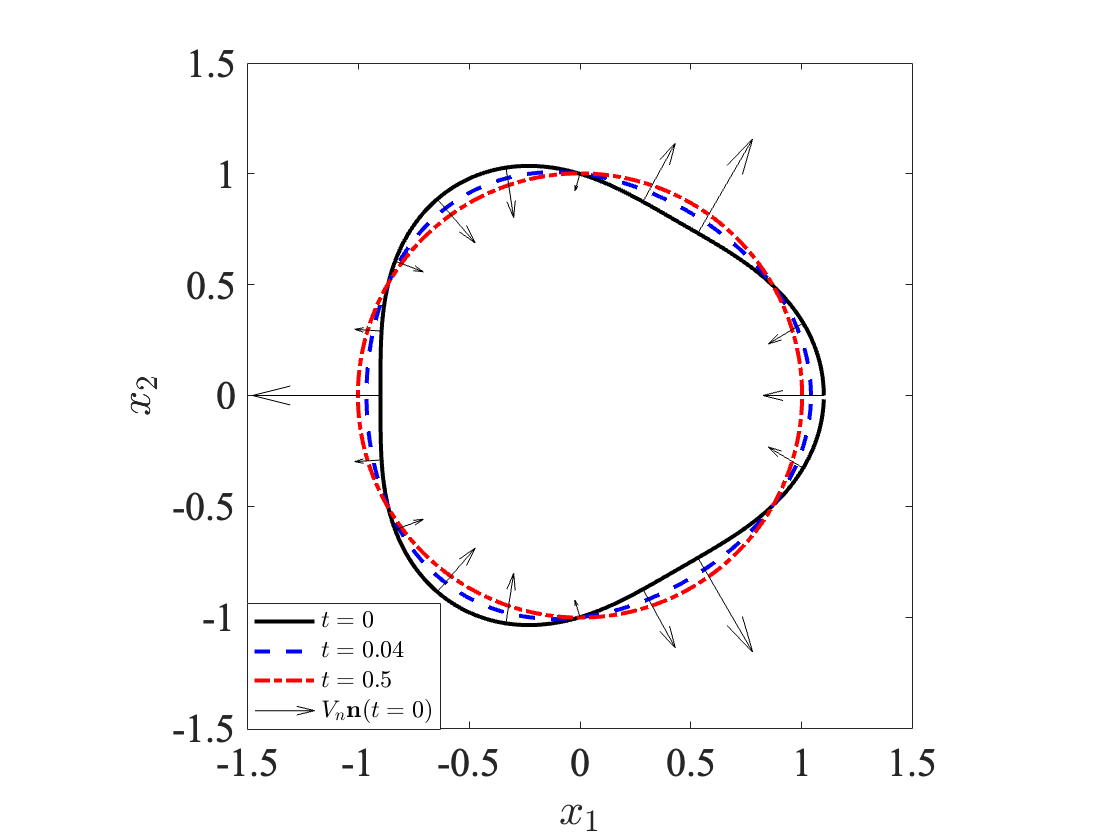}}
\subfigure[Radius evolution for $D_1=0.1$, $D_2=3$ compared with steady-state value.]{\includegraphics[width=0.45\textwidth]{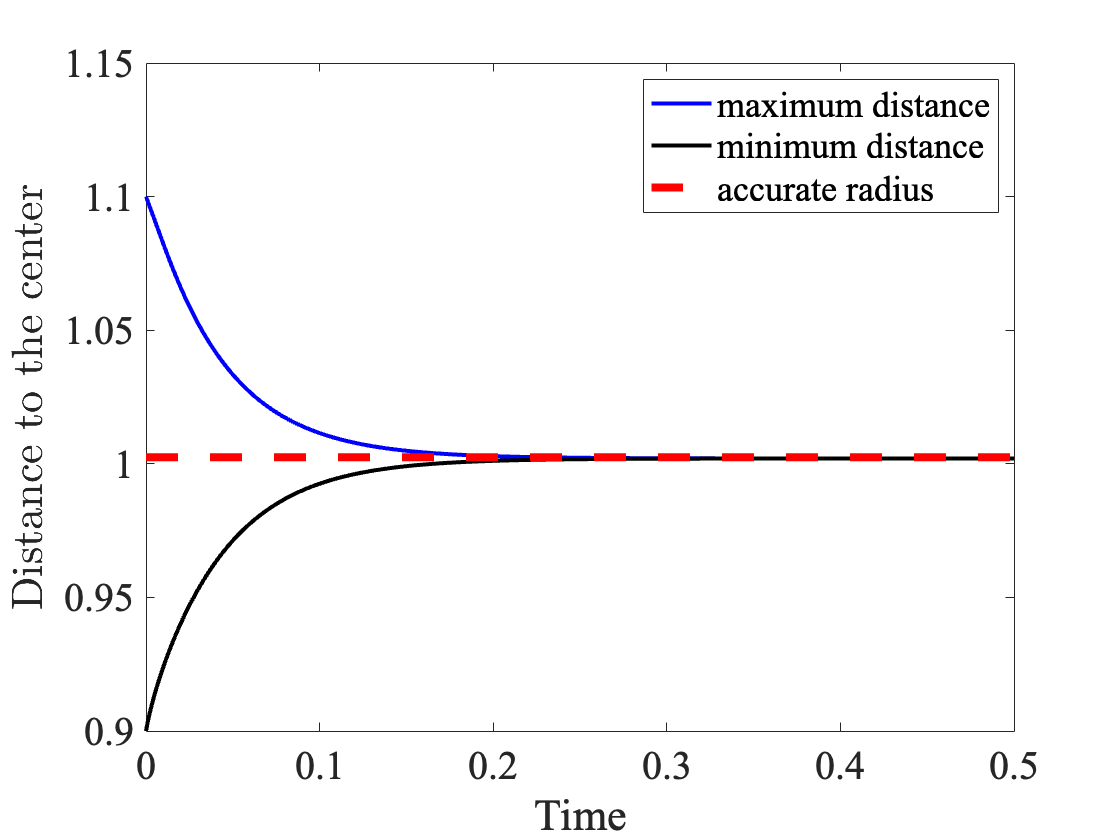}}
\subfigure[Evolution dynamics for $D_1=0.1$, $D_2=5$.]{\includegraphics[width=0.45\textwidth]{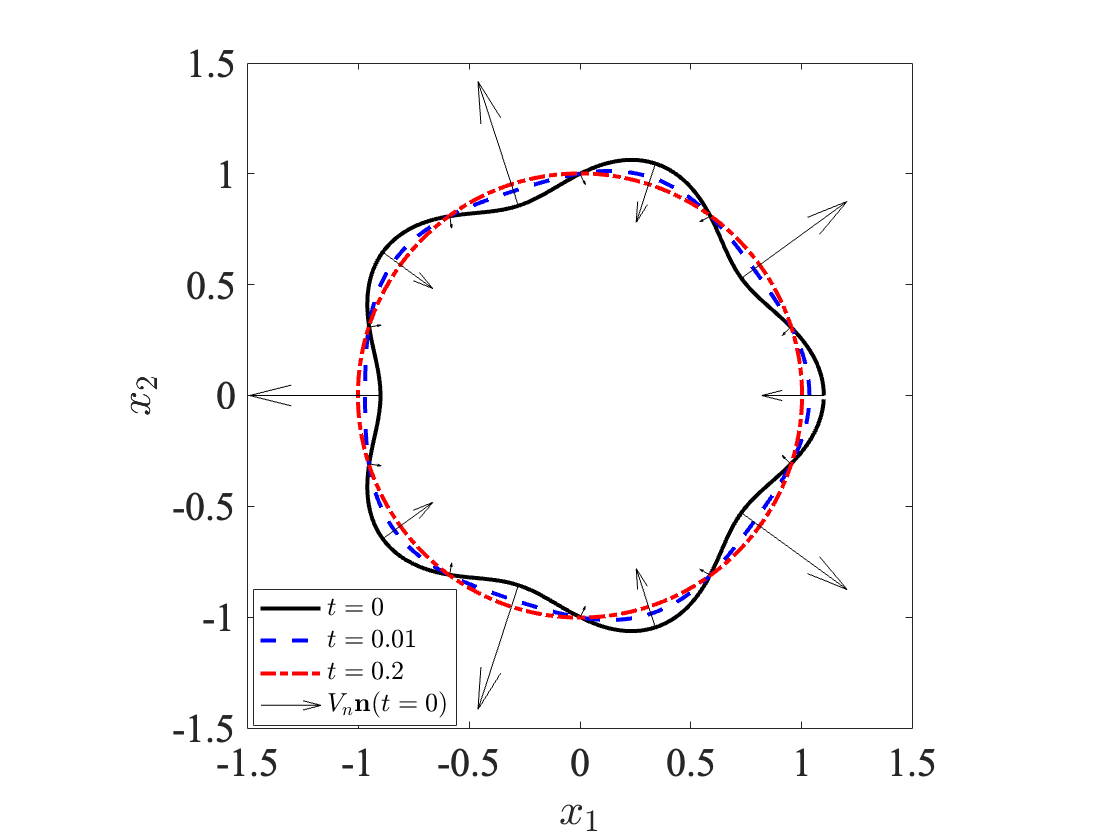}}
\subfigure[Radius evolution for $D_1=0.1$, $D_2=5$ compared with steady-state value.]{\includegraphics[width=0.45\textwidth]{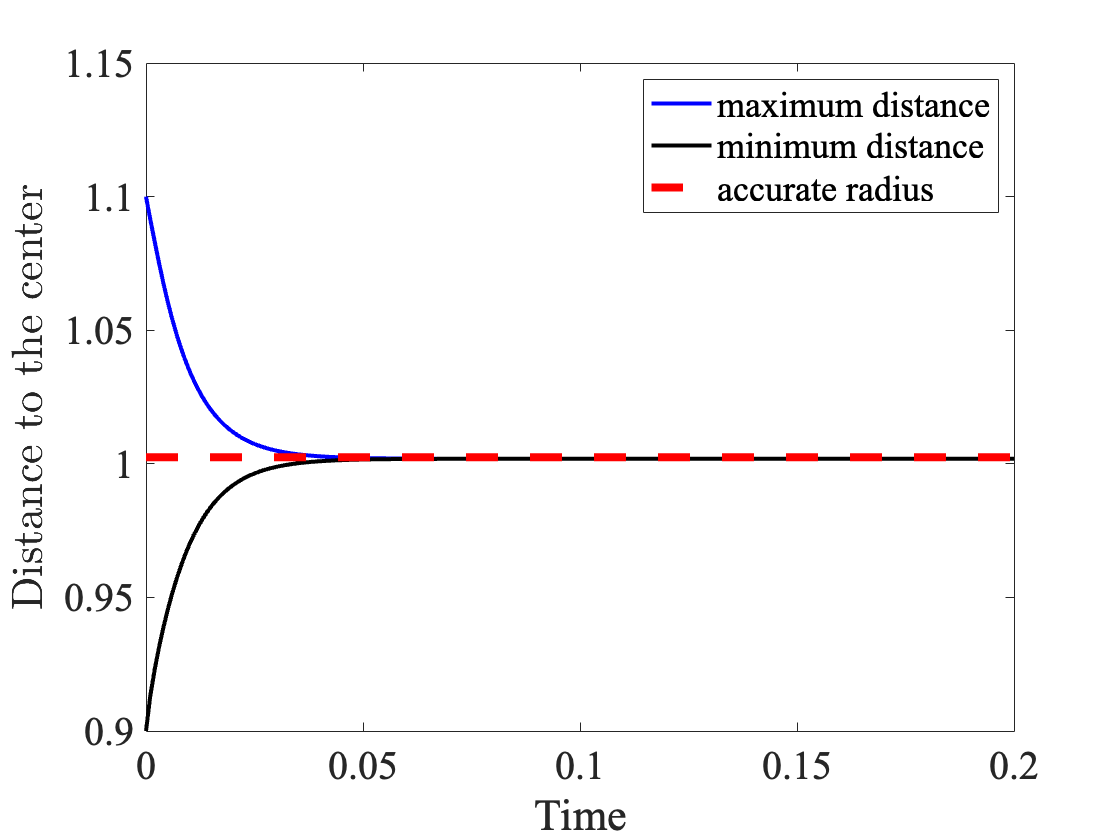}}
\caption{The evolution dynamics for perturbed circles with different $D_1$ and $D_2$. In (a), and (c), the profiles of $\mathbf x(t)$ at different times and boundary motion $ V_n \mathbf n$ at $t=0$ for different $D_1$ and $D_2$ are shown. In (b) and (d), comparisons of the maximum and minimum distances from the boundary points to the center with the steady-state circle radius $r_s$ for different $D_1$ and $D_2$ are shown, where we observe that the perturbed circles with different $D_1$ and $D_2$ evolve into circles with radius $r_s$. }
\label{fig:perturbed-circle2}
\end{figure}

\bibliographystyle{plain}  

\end{document}